\documentclass[11pt, a4paper, oneside]{amsart}

\usepackage[english]{babel}
\usepackage{amsmath, amsthm, amsfonts, mathrsfs, amssymb}
\usepackage{mathtools}
\mathtoolsset{centercolon}
\usepackage{booktabs}
\usepackage[shortlabels]{enumitem}
\usepackage[colorlinks, citecolor = blue]{hyperref}

\setlist[itemize]{leftmargin=25pt}
\setlist[enumerate]{leftmargin=25pt}

\newcommand{\N}{\ensuremath{\mathbb{N}}}
\newcommand{\Z}{\ensuremath{\mathbb{Z}}}

\newcommand{\R}{\ensuremath{\mathbb{R}}}
\newcommand{\C}{\ensuremath{\mathbb{C}}}
\newcommand{\K}{\ensuremath{\mathbb{K}}}
\newcommand{\E}{\ensuremath{\mathbb{E}}}
\renewcommand{\P}{\ensuremath{\mathbb{P}}}

\newcommand{\mbs}{\boldsymbol}

\newcommand{\mc}{\mathcal}
\newcommand{\ms}{\mathscr}
\DeclareMathAlphabet{\mathpzc}{OT1}{pzc}{m}{it}
\newcommand{\mz}{\mathpzc}

\DeclarePairedDelimiter\abs{\lvert}{\rvert}

\DeclarePairedDelimiter\cbrace\{\}
\DeclarePairedDelimiter\ha()
\DeclarePairedDelimiter{\ip}\langle\rangle
\DeclarePairedDelimiter{\nrm}\lVert\rVert

\DeclarePairedDelimiter{\ceil}{\lceil}{\rceil}

\newcommand{\nrmb}[1]{\bigl\|#1\bigr\|}

\newcommand{\hab}[1]{\bigl(#1\bigr)}
\newcommand{\cbraceb}[1]{\bigl\{#1\bigr\}}
\newcommand{\ipb}[1]{\bigl\langle#1\bigr\rangle}

\newcommand{\nrms}[1]{\Bigl\|#1\Bigr\|}

\newcommand{\has}[1]{\Bigl(#1\Bigr)}
\newcommand{\cbraces}[1]{\Bigl\{#1\Bigr\}}

\newcommand{\bracs}[1]{\Bigl[#1\Bigr]}

\DeclareMathOperator{\loc}{loc}
\DeclareMathOperator{\supp}{supp}
\DeclareMathOperator{\ind}{\mathbf{1}}
\DeclareMathOperator{\UMD}{UMD}

\DeclareMathOperator{\BMO}{BMO}

\DeclareMathOperator{\RMF}{RMF}
\DeclareMathOperator{\Rad}{Rad}
\DeclareMathOperator{\Lat}{Lat}
\DeclareMathOperator*{\esssup}{ess\,sup}

\DeclareMathOperator{\diam}{diam}
\DeclareMathOperator{\Dini}{Dini}

\newcommand{\dd}{\hspace{2pt}\mathrm{d}}
\newcommand{\ddn}{\mathrm{d}}
\newcommand{\ee}{\mathrm{e}}
\renewcommand{\emptyset}{\varnothing}

\def\avint_#1{\mathchoice{\mathop{\kern 0.2em\vrule width 0.6em height 0.69678ex depth -0.58065ex \kern -0.8em \intop}\nolimits_{\kern -0.4em#1}}{\mathop{\kern 0.1em\vrule width 0.5em height 0.69678ex depth -0.60387ex \kern -0.6em \intop}\nolimits_{#1}} {\mathop{\kern 0.1em\vrule width 0.5em height 0.69678ex depth -0.60387ex \kern -0.6em \intop}\nolimits_{#1}} {\mathop{\kern 0.1em\vrule width 0.5em height 0.69678ex depth -0.60387ex \kern -0.6em \intop}\nolimits_{#1}}}

\DeclareFontFamily{U}{mathx}{\hyphenchar\font45}
\DeclareFontShape{U}{mathx}{m}{n}{<5> <6> <7> <8> <9> <10> <10.95> <12> <14.4> <17.28> <20.74> <24.88> mathx10}{}
\DeclareSymbolFont{mathx}{U}{mathx}{m}{n}
\DeclareFontSubstitution{U}{mathx}{m}{n}
\DeclareMathAccent{\widecheck}{0}{mathx}{"71}

\newtheorem{theorem}{Theorem}
\newtheorem{corollary}[theorem]{Corollary}
\newtheorem{lemma}[theorem]{Lemma}
\newtheorem{proposition}[theorem]{Proposition}

\theoremstyle{remark}
\newtheorem{remark}[theorem]{Remark}

\theoremstyle{definition}
\newtheorem{definition}[theorem]{Definition}

\numberwithin{theorem}{section}
\numberwithin{equation}{section}

\allowdisplaybreaks

\title[On pointwise $\ell^r$-sparse domination]{On pointwise $\ell^r$-sparse domination in a space of homogeneous type}
\author{Emiel Lorist}
\thanks{The author is supported by the VIDI subsidy 639.032.427 of the Netherlands Organization for Scientific Research (NWO)}

\address{Delft Institute of Applied Mathematics \\ Delft University of Technology \\ P.O. Box 5031\\ 2600 GA Delft \\The Netherlands}
\email{e.lorist@tudelft.nl}

\begin{document}

\begin{abstract}
We prove a general sparse domination theorem in a space of homogeneous type, in which a vector-valued operator is controlled pointwise by a positive, local expression called a sparse operator. We use the structure of the operator to get sparse domination in which the usual $\ell^1$-sum in the sparse operator is replaced by an $\ell^r$-sum.

This sparse domination theorem is applicable to various operators from both harmonic analysis and (S)PDE. Using our main theorem, we prove the $A_2$-theorem for vector-valued Calder\'on--Zygmund operators in a space of homogeneous type, from which we deduce an anisotropic, mixed norm Mihlin multiplier theorem. Furthermore we show quantitative weighted norm inequalities for the Rademacher maximal operator, for which Banach space geometry plays a major role.
\end{abstract}



\maketitle

\section{Introduction}
The technique of controlling various operators by so-called sparse operators has proven to be a very useful tool to obtain (sharp) weighted norm inequalities in the past decade.
 The key feature in this approach is that a typically signed and non-local operator is dominated, either in norm, pointwise or in dual form, by a positive and local expression.

  The sparse domination technique comes from Lerner's work towards an alternative proof of the $A_2$-theorem, which was first proven by Hyt\"onen in \cite{Hy12}. In \cite{Le13a} Lerner applied his local mean oscillation decomposition approach to the $A_2$-theorem, estimating the norm of a Calder\'on-Zygmund operator by the norm of a sparse operator. This was later improved to a pointwise estimate independently by Conde-Alonso and Rey  \cite{CR16} and by Lerner and Nazarov \cite{LN15}. Afterwards, Lacey \cite{La17b} obtained the same result for a slightly larger class of Calder\'on-Zygmund operators by a stopping cube argument instead of the local mean oscillation decomposition approach. This argument was further refined by Hyt\"onen, Roncal and Tapiola  \cite{HRT17} and afterwards made strikingly clear by Lerner \cite{Le16}, where the following abstract sparse domination principle was shown:

 If $T$ is a bounded sublinear operator from $L^{p_1}(\R^n)$ to $L^{p_1,\infty}(\R^n)$ and the \emph{grand maximal truncation operator}
\begin{equation*}
  \mc{M}_T f(s) := \sup_{Q \ni s} \esssup_{s' \in Q} \,\abs{T(f \ind_{\R^n\setminus 3Q})(s')}, \qquad s \in \R^n
\end{equation*}
is bounded from $L^{p_2}(\R^n)$ to $L^{p_2,\infty}(\R^n)$ for some $1 \leq p_1,p_2 <\infty$, then there is an $\eta \in (0,1)$ such that for every compactly supported $f \in L^{p}(\R^n)$ with $p_0 := \max\cbrace{p_1, p_2}$ there exists an $\eta$-sparse family of cubes $\mc{S}$ such that
\begin{equation}\label{eq:sparsepointwiseintro}
  \abs{Tf(s)} \lesssim
  \sum_{Q \in \mc{S}} \ip{\abs{f}}_{p_0,Q}\ind_Q(s),\qquad  s \in \R^n.
\end{equation}
Here $\ip{{f}}_{p,Q}^p:= \avint_Q {f}^p :=\frac{1}{\abs{Q}}\int_Q{f}^p$ for $p \in (0,\infty)$ and positive $f \in L^p_{\loc}(\R^n)$ and we call a family of cubes $\mc{S}$ $\eta$-sparse if for every $Q \in \mc{S}$ there exists a measurable set $E_Q \subseteq Q$ such that $\abs{E_Q} \geq \eta\abs{Q}$ and such that the $E_Q$'s are pairwise disjoint.

 This sparse domination principle was further generalized in the recent paper \cite{LO19} by Lerner and Ombrosi, in which the authors  showed that the weak $L^{p_2}$-boundedness of the more flexible operator
\begin{equation*}
  \mc{M}_{T,\alpha}^{\#} f(s) := \sup_{Q \ni s} \esssup_{s',s'' \in Q} \,\abs{T(f \ind_{\R^n\setminus \alpha Q})(s') - T(f \ind_{\R^n\setminus \alpha Q})(s'')}, \qquad s \in \R^n
\end{equation*}
for some $\alpha \geq 3$ is already enough to deduce the pointwise sparse domination as in \eqref{eq:sparsepointwiseintro}. Furthermore, they relaxed the weak $L^{p_1}$-boundedness condition on $T$ to a condition in the spirit of the $T(1)$-theorem.

\subsection{Main result}
Our main result is a generalization of the main result in \cite{LO19} in the following four directions:
\begin{enumerate}[(i)]
  \item \label{it:generalizationhom} We replace $\R^n$ by a space of homogeneous type $(S,d,\mu)$.
  \item \label{it:generalizationvect} We let $T$ be an operator from $L^{p_1}(S;X)$ to $L^{p_1,\infty}(S;Y)$, where $X$ and $Y$ are Banach spaces.
  \item \label{it:generalizationsub} We use structure of the operator $T$ and geometry of the Banach space $Y$ to replace the $\ell^1$-sum in the sparse operator by an $\ell^r$-sum for $r \geq 1$.
  \item \label{it:generalizationlocal} We replace the truncation $T(f \ind_{\R^n\setminus \alpha Q})$ in the grand maximal truncation operator by an abstract localization principle.
\end{enumerate}
The extensions \ref{it:generalizationhom} and \ref{it:generalizationvect} are relatively straightforward. The main novelty of this paper is \ref{it:generalizationsub}, which controls the weight characteristic dependence that can be deduced from the sparse domination. Generalization \ref{it:generalizationlocal} will only make its appearance in Theorem \ref{theorem:localsparse} and can be used to make the associated grand maximal truncation operator easier to estimate in specific situations.

\bigskip

Let $(S,d,\mu)$ be a space of homogeneous type
and let $X$ and $Y$ be Banach spaces.
For a bounded linear operator $T$ from $L^{p_1}(S;X)$ to $L^{p_1,\infty}(S;Y)$ and $\alpha\geq 1$ we define the
following \emph{sharp grand maximal truncation operator}
\begin{align*}
  \mc{M}_{T,\alpha}^{\#}&f(s) :=\sup_{B\ni s} \esssup_{s',s''\in B}\, \nrmb{T(f\ind_{S\setminus {\alpha B}})(s')-T(f\ind_{S\setminus {\alpha B}})(s'')}_Y,\qquad s \in S,
\end{align*}
where the supremum is taken over all balls $B\subseteq S$ containing $s\in S$.  Our main theorem reads as follows.

\begin{theorem}\label{theorem:main}
Let $(S,d,\mu)$ be a space of homogeneous type and let $X$ and $Y$ be Banach spaces. Take $p_1,p_2,r \in [1,\infty)$ and set $p_0:=\max\cbrace{p_1, p_2}$. Take $\alpha \geq 3c_d^2/\delta$, where $c_d$ is the quasi-metric constant and $\delta$ is as in Proposition \ref{proposition:dyadicsystem}. Assume the following conditions:
\begin{itemize}
  \item $T$ is a bounded linear operator from $L^{p_1}(S;X)$ to $L^{p_1,\infty}(S;Y)$.
  \item $\mc{M}_{T,\alpha}^{\#}$ is a bounded operator from $L^{p_2}(S;X)$ to $L^{p_2,\infty}(S)$.
  \item There is a $C_r>0$ such that for disjointly and boundedly supported $f_1,\ldots,f_n \in L^{p_0}(S;X)$
\begin{equation*}
  \qquad \nrms{T\hab{\sum_{k=1}^n f_k}(s)}_Y\leq C_r \, \has{\sum_{k=1}^n \nrmb{Tf_k(s)}_Y^r}^{1/r},
  \qquad s \in S.
\end{equation*}
\end{itemize}
 Then there is an $\eta \in (0,1)$ such that for any boundedly supported $f \in L^{p_0}(S;X)$ there is an $\eta$-sparse collection of cubes $\mc{S}$ such that
\begin{align*}
   \nrm{ Tf(s)}_Y&\lesssim_{S,\alpha} C_T \,C_r\, \has{\sum_{Q \in \mc{S}} \ipb{\nrm{f}_X}_{p_0,Q}^r \ind_{Q}(s)}^{1/r}, \qquad s \in S,
\end{align*}
where $C_T={\nrm{T}_{L^{p_1}\to L^{p_1,\infty}} + \nrm{\mc{M}_{T,\alpha}^{\#}}_{L^{p_2}\to L^{p_2,\infty}}}$.
\end{theorem}
As the assumption in the third bullet of Theorem \ref{theorem:main} expresses a form of sublinearity of the operator $T$ when $r=1$, we will call this assumption \emph{$r$-sublinearity}. Note that it is crucial that the constant $C_r$  is independent of $n \in \N$. If $C_r = 1$ it suffices to consider $n=2$.

\subsection{Sharp weighted norm inequalities}
One of the main reasons to study sparse domination of an operator is the fact that sparse bounds yield weighted norm inequalities and these weighted norm inequalities are sharp for many operators. Here sharpness is meant in the sense that for $p \in (p_0,\infty)$ we have  a $\beta\geq 0$ such that
\begin{equation}\label{eq:sharp}
  \nrm{T}_{L^p(S,w;X) \to L^p(S,w;Y)} \lesssim [w]_{A_{p/p_0}}^\beta, \qquad w \in A_{p/p_0}
\end{equation}
and \eqref{eq:sharp} is false for any $\beta' <\beta$.

The first result of this type was obtained by Buckley \cite{Bu93}, who showed that $\beta = \tfrac{1}{p-1}$ for the Hardy--Littlewood maximal operator. A decade later, the quest to find sharp weighted bounds attracted renewed attention because of the work of Astala, Iwaniec and Saksman \cite{AIS01}. They proved sharp regularity results for the solution to the Beltrami equation under the assumption that $\beta = 1$ for the Beurling--Ahlfors transform  for $p\geq 2$. This linear dependence on the $A_p$ characteristic for the Beurling--Ahlfors transform was shown by Petermichl and Volberg in \cite{PV02}.  Another decade later, after many partial results, sharp weighted norm inequalities were obtained for general Calder\'on--Zygmund operators by Hyt\"onen in \cite{Hy12} as discussed before.

In Section \ref{section:weights} we will prove weighted $L^p$-boundedness for the sparse operators appearing in Theorem \ref{theorem:main}. As a direct corollary from  Theorem \ref{theorem:main} and Proposition \ref{proposition:weights} we have:

\begin{corollary}\label{corollary:main}
  Under the assumptions of Theorem \ref{theorem:main} we have for all $p \in (p_0,\infty)$ and $w \in A_{p/p_0}$
\begin{align*}
  \nrm{T}_{L^p(S,w;X)\to L^p(S,w;Y)} &\lesssim C_T\,C_r\, [w]_{A_{p/p_0}}^{\max\cbraceb{\frac{1}{p-p_0},\frac{1}{r}}},
\end{align*}
where the implicit constant depends on $S,p_0,p,r$ and $\alpha$.
\end{corollary}
As noted before the main novelty in Theorem \ref{theorem:main} is the introduction of the parameter $r \in [1,\infty)$. The $r$-sublinearity assumption in Theorem \ref{theorem:main} becomes more restrictive as $r$ increases and the conclusions of Theorem \ref{theorem:main} and Corollary \ref{corollary:main} consequently become stronger.  In order to check whether the dependence on the weight characteristic is sharp, one can employ e.g. \cite[Theorem 1.2]{LPR15}, which provides a lower bound for the best possible weight characteristic dependence in terms of the operator norm of $T$ from $L^p(S;X)$ to $L^p(S;Y)$. For some operators, like Littlewood--Paley or maximal operators, sharpness in the estimate in Corollary \ref{corollary:main} is attained for $r>1$ and thus Theorem \ref{theorem:main} can be used to show sharp weighted bounds for more operators than precursors like \cite[Theorem 1.1]{LO19}.

\subsection{How to apply our main result}
Let us outline the typical way how one applies Theorem \ref{theorem:main} (or the local and more general version in Theorem \ref{theorem:localsparse}) to obtain (sharp) weighted $L^p$-boundedness for an operator $T$:
\begin{enumerate}[(i)]
  \item If $T$ is not linear it is often \emph{linearizable}, which means that we can linearize it by putting part of the operator in the norm of the Banach space $Y$. For example if $T$ is a Littlewood--Paley square function we take $Y=L^2$ and if $T$ is a maximal operator we take $Y=\ell^\infty$. Alternatively one can apply Theorem \ref{theorem:localsparse}, which is a local and more abstract version of Theorem \ref{theorem:main} that does not assume $T$ to be linear.
  \item The weak $L^{p_1}$-boundedness of $T$ needs to be studied separately and is often already available in the literature.
  \item The operator $\mc{M}_{T,\alpha}^{\#}$ reflects the non-localities of the operator $T$. The weak $L^{p_2}$-boundedness of $\mc{M}_{T,\alpha}^{\#}$ requires an intricate study of the structure of the operator. In many examples $\mc{M}_{T,\alpha}^{\#}$ can be pointwise dominated by the Hardy--Littlewood maximal operator $M_{p_2}$, which is weak $L^{p_2}$-bounded. This is exemplified for Calder\'on--Zygmund operators in the proof of Theorem \ref{theorem:A2}.
      Sometimes one can choose a suitable localization in Theorem \ref{theorem:localsparse} such that the sharp maximal truncation operator is either zero (see Section \ref{section:maximal} on the Rademacher maximal operator), or pointwise dominated by $T$.
  \item The $r$-sublinearity assumption on $T$ is trivial for $r=1$, which suffices if one is not interested in quantitative weighted bounds. To check the $r$-sublinearity for some $r>1$ one needs to use the structure of the operator and often also the geometric properties of the Banach space $Y$ like type $r$.
      See, for example, the proofs of Theorems \ref{theorem:RMF} and \cite[Theorem 6.4]{LV19} how to check $r$-sublinearity in concrete cases.
\end{enumerate}

\subsection{Applications}
The main motivation to generalize the results in \cite{LO19} comes from the application in the recent work \cite{LV19} by Veraar and the author, in which Calder\'on--Zygmund theory is developed for stochastic singular integral operators. In particular, in \cite[Theorem 6.4]{LV19} Theorem \ref{theorem:main} is applied with $p_1=p_2=r=2$ to prove a stochastic version of the vector-valued $A_2$-theorem for Calder\'on--Zygmund operators, which yields new results in the theory of maximal regularity for stochastic partial differential equations. The fact that $r=2$ in \cite[Theorem 6.4]{LV19} was needed to obtain a sharp result motivated the introduction of the parameter $r$ in this paper.
In future work, further applications of Theorem \ref{theorem:main} to both deterministic and stochastic partial differential equations will be given, for which it is crucial that we allow spaces of homogeneous type instead of just $\R^n$, as in these applications $S$ is typically $\R_+\times \R^{n}$ with the parabolic metric.

In this paper we will focus on applications in harmonic analysis. We will  provide a few examples that illustrate the sparse domination principle nicely, and comment on further potential applications in Section \ref{section:further}.
\begin{itemize}
  \item As a first application of Theorem \ref{theorem:main} we prove an $A_2$-theorem for vector-valued Calder\'on--Zygmund operators with operator-valued kernel in a space of homogeneous type. The $A_2$-theorem for vector-valued Calder\'on--Zygmund operators with operator-valued kernel in Euclidean space has previously been proven in \cite{HH14} and the $A_2$-theorem for scalar-valued Calder\'on--Zygmund operators in spaces of homogeneous type in \cite{NRV13,AV14}. Our theorem unifies these two results.
  \item Using the $A_2$-theorem, we prove a weighted, anisotropic, mixed norm Mihlin multiplier theorem, which is a natural supplement to the recent results in \cite{FHL18} and is particularly useful in the study of spaces of smooth, vector-valued functions.
  \item  In our second application of Theorem \ref{theorem:main} we study sparse domination and quantitative weighted norm inequalities for the Rademacher maximal operator, extending the qualitative bounds in Euclidean space in \cite{Ke13}. The proof  demonstrates how one can use the geometry of the Banach space to deduce $r$-sublinearity for an operator. As a corollary, we deduce that the lattice Hardy--Littlewood and the Rademacher maximal operator are not comparable.
\end{itemize}

\subsection{Outline}
This paper is organized as follows: After introducing spaces of homogeneous type and dyadic cubes in such spaces in Section \ref{section:SHT}, we will set up our abstract sparse domination framework and deduce Theorem \ref{theorem:main} in Section  \ref{section:main}. We also give some further generalizations of our main results. In Section \ref{section:weights} we introduce weights and state weighted bounds for the sparse operators in the conclusions of Theorem \ref{theorem:main}, from which Corollary \ref{corollary:main} follows. To prepare for our application sections, we will discuss some preliminaries on e.g. Banach space geometry in Section \ref{section:Banachspace}. Afterwards we will
use our main result to prove the previously discussed applications in Sections \ref{section:A2}-\ref{section:maximal}.
Finally, in Section \ref{section:further} we discuss some potential further applications of our main result.

\section{Spaces of homogeneous type}\label{section:SHT}
A space of homogeneous type $(S,d,\mu)$, originally introduced by Coifman and Weiss in \cite{CW71}, is a set $S$ equipped with a quasi-metric $d$ and a doubling Borel measure $\mu$. That is,
 a metric $d$  which instead of the triangle inequality satisfies
\begin{equation*}
  d(s,t) \leq c_d\, \hab{d(s,u)+d(u,t)}, \qquad s,t,u\in S
\end{equation*}
for some $c_d\geq 1$, and a Borel measure $\mu$ that satisfies the doubling property
\begin{equation*}
  \mu\hab{B(s,2\rho)} \leq c_\mu \,\mu\hab{B(s,\rho)}, \qquad s \in S,\quad \rho>0
\end{equation*}
for some $c_\mu\geq 1$, where $B(s,\rho):=\cbrace{t \in S:d(s,t)<\rho}$ is the ball around $s$ with radius $\rho$. Throughout this paper we will assume additionally that all balls $B\subseteq S$ are Borel sets and that we have $0 <\mu(B)<\infty$.

It was shown in \cite[Example 1.1]{St15} that it can indeed happen that balls are not Borel sets in a quasi-metric space. This can be circumvented by taking topological closures and adjusting the constants $c_d$ and $c_\mu$ accordingly. However, to simplify matters we just assume all balls to be Borel sets and leave the necessary modifications if this is not the case to the reader.
The size condition on the measure of a ball ensures that taking the average $\ip{{f}}_{p,B}$ of a positive function $f \in L^p_{\loc}(S)$  over a ball $B\subseteq S$ is always well-defined.

As $\mu$ is a Borel measure, i.e. a measure defined on the Borel $\sigma$-algebra of the quasi-metric space $(S,d)$, the Lebesgue differentiation theorem holds and as a consequence the continuous functions with bounded support are dense in $L^p(S)$ for all $p \in [1,\infty)$. The Lebesgue differentiation theorem and consequently our results remain valid if $\mu$ is a measure defined on a $\sigma$-algebra $\Sigma$ that contains the Borel $\sigma$-algebra as long as the measure space $(S,\Sigma,\mu)$ is Borel semi-regular. See \cite[Theorem 3.14]{AM15} for the details.

Throughout we will write that an estimate depends on $S$ if it depends on $c_d$ and $c_\mu$. For a thorough introduction to and a list of examples of spaces of homogeneous type we refer to the monographs of Christ \cite{Ch90} and Alvarado and Mitrea \cite{AM15}.

\subsection{Dyadic cubes}
Let $0<c_0\leq C_0 <\infty$ and $0<\delta<1$. Suppose that for $k \in \Z$ we have an index set $J_k$, pairwise disjoint collection $\ms{D}_k = \cbrace{Q^j_k}_{j \in J_k}$ of measurable sets and a collection of points $\cbrace{z_j^k}_{j \in J_k}$. We call $\ms{D} := \bigcup_{k \in \Z}\ms{D}_k$ a \emph{dyadic system} with parameters $c_0$, $C_0$ and $\delta$ if it satisfies the following properties:
\begin{enumerate}[(i)]
\item For all $k \in \Z$ we have $$S = \bigcup_{j \in J_k} Q_j^k;$$
\item \label{it:dyadicsystem2}For $k \geq l$, $Q \in \ms{D}_k$ and $Q' \in \ms{D}_l$ we either have $Q\cap Q' =\emptyset$ or $Q \subseteq Q'$;
\item \label{it:dyadicsystem3}For each $k \in \Z$ and $j \in J_k$ we have
\begin{equation*}
  B(z_j^k,c_0 \delta^k) \subseteq Q_j^k \subseteq B(z_j^k, C_0 \delta^k);
\end{equation*}
\end{enumerate}
We will call the elements of a dyadic system $\ms{D}$ cubes and for a cube $Q \in \ms{D}$ we define the \emph{restricted dyadic system} $\ms{D}(Q):=\cbrace{P\in \ms{D}:P\subseteq Q}$. We will say that an estimate depends on $\ms{D}$ if it depends on the parameters $c_0$, $C_0$ and $\delta$.

One can view $z_j^k$ and $\delta^k$ as the center and side length of a cube $Q_j^k \in \ms{D}_k$. These have to be with respect to a specific $k \in \Z$, as this $k$ may not be unique. We therefore think of a cube $Q \in \ms{D}$ to also encode the information of its center $z$ and generation $k$.
The structure of individual dyadic cubes $Q \in \ms{D}$ in a space of homogeneous type can be very messy and consequently the dilations of such cubes do not have a canonical definition. Therefore for a cube $Q \in \ms{D}$ with center $z$ and of generation $k$ we define the \emph{dilations $\alpha Q$} for $\alpha\geq 1$ as
\begin{equation*}
  \alpha Q := B\hab{z,\alpha \cdot C_0 \delta^k},
\end{equation*}
which are actually dilations of the ball that contains $Q$ by property \ref{it:dyadicsystem3} of a dyadic system.
\bigskip

When $S=\R^n$ and $d$ is the Euclidean distance, the standard dyadic cubes form a dyadic system and, combined with its translates over $\alpha \in \cbrace{0,\frac{1}{3},\frac23}^n$, it holds that any ball in $\R^n$ is contained in a cube of comparable size from one of these dyadic systems (see e.g. \cite[Lemma 3.2.26]{HNVW16}). We will rely on the following proposition for the existence of dyadic systems with this property in a general space of homogeneous type. For the proof and a more detailed discussion we refer to \cite{HK12}.

\begin{proposition}\label{proposition:dyadicsystem}
  Let $(S,d,\mu)$ be a space of homogeneous type. There exist $0<c_0\leq C_0<\infty$, $\gamma\geq 1$, $0<\delta<1$ and $m \in \N$ such that there are dyadic systems $\ms{D}^1,\ldots,\ms{D}^m$ with parameters $c_0$, $C_0$ and $\delta$, and with the property that for each $s \in S$ and $\rho>0$ there is a $j\in \cbrace{1,\ldots,m}$ and a $Q \in \ms{D}^j$  such that $$B(s,\rho) \subseteq Q, \qquad\text{and}\qquad  \diam(Q) \leq \gamma \rho.$$
\end{proposition}

The following covering lemma will be used in the proof of our main theorem:
\begin{lemma}\label{lemma:covering}
Let $(S,d,\mu)$ be a space of homogeneous type and $\ms{D}$ a dyadic system with parameters $c_0$, $C_0$ and $\delta$. Suppose that $\diam(S) = \infty$, take  $\alpha \geq {3c_d^2/\delta}$ and let $E \subseteq S$ satisfy $0<\diam(E)<\infty$. Then there exists a partition $\mc{D} \subseteq \ms{D}$ of $S$ such that $E \subseteq \alpha Q$ for all $Q \in \mc{D}$.
\end{lemma}

\begin{proof}
  For $s \in S$ and $k \in \Z$ let $Q_s^k\in \ms{D}_k$ be the unique cube such that $s \in Q_s^k$ and denote its center by $z_{s}^k$. Define
  \begin{equation*}
    K_s := \cbraceb{k \in \Z: E \not\subseteq 2 c_d Q_s^k},
  \end{equation*}
  where $c_d$ is the quasi-metric constant.
  If $k \in \Z$ is such that
  $$\diam(2c_dQ_s^k) \leq 4c_d^2C_0\delta^k < \diam(E),$$ then $E\not\subseteq  2c_dQ_s^k$, i.e. $k \in K_s$ so is $K_s$ non-empty. On the other hand if $k \in \Z$ is such that
$
    C_0\delta^k > \sup_{s' \in E} d(s,s'),
$
  then
  \begin{equation*}
 \sup_{s'\in E} d(s',z_{s}^k) \leq c_d \,\hab{\sup_{s' \in E} d(s,s') + d(s, z_{s}^k)} \leq 2c_dC_0\delta^k
  \end{equation*}
  so $E \subseteq 2c_dQ_s^k$ and thus $k \notin K_s$. Therefore $K_s$ is bounded from below.

  Define $k_s := \min K_s$ and set $\mc{D}:= \cbrace{Q_s^{k_s}:s \in S}$. Then $\mc{D}$ is a partition of $S$. Indeed, suppose that for $s,s'\in S$ we have $Q_s^{k_s} \cap Q_{s'}^{k_{s'}}\neq \emptyset$. Then using property \ref{it:dyadicsystem2} of a dyadic system we may assume without loss of generality that $Q_s^{k_s} \subseteq Q_{s'}^{k_{s'}}$. Property \ref{it:dyadicsystem2} of a dyadic system then implies that $k_s \geq k_{s'}$. In particular $s \in Q_{s'}^{k_{s'}}$, so by the minimality of $k_s$ we must have $k_s = k_{s'}$. Therefore since the elements of $\ms{D}_{k_s}$ are pairwise disjoint we can conclude $Q_s^{k_s} = Q_{s'}^{k_{s'}}$.

  To conclude note that $z_s^{k_s} \in Q_s^{k_s} \subseteq Q_s^{k_s-1}$ by property \ref{it:dyadicsystem2} of a dyadic system, so $d(z_s^{k_s-1},z_s^{k_s}) \leq C_0\delta^{k_s-1}$. Therefore using the minimality of $k_s$ we obtain
  $$E \subseteq 2c_dQ_s^{k_s-1} = B(z_s^{k_s-1},2c_dC_0\delta^{k_s-1}) \subseteq B\has{z_s^{k_s},\frac{3c_d^2}{\delta} \cdot C_0\delta^{k_s}} \subseteq \alpha Q_s^{k_s},$$
  which finishes the proof.
\end{proof}

\subsection{The Hardy--Littlewood maximal operator}\label{subsection:HL}
On a space of homogeneous type  $(S,d,\mu)$  with a dyadic system $\ms{D}$ we define the \emph{dyadic Hardy--Littlewood maximal operator} for $f \in L^1_{\loc}(S)$ by
\begin{equation*}
  M^{\ms{D}}f(s):= \sup_{Q\in \ms{D}: s \in Q} \ipb{\abs{f}}_{1,Q}, \qquad s \in S.
\end{equation*}
By Doob's maximal inequality (see e.g. \cite[Theorem 3.2.2]{HNVW16}) $M^{\ms{D}}$ is strong $L^p$-bounded for all $p\in (1,\infty)$ and weak $L^1$-bounded. We define the (non-dyadic) \emph{Hardy--Littlewood maximal operator} for $f \in L^1_{\loc}(S)$ by
\begin{equation*}
  Mf(s):= \sup_{B \ni s} \,\ipb{\abs{f}}_{1,Q}, \qquad s \in S,
\end{equation*}
where the supremum is taken over all balls $B \subseteq S$ containing $s$. By Proposition \ref{proposition:dyadicsystem} there are dyadic systems $\ms{D}^1,\ldots,\ms{D}^m$ such that
\begin{equation*}
  Mf(s) \lesssim_S \sum_{j=1}^m M^{\ms{D}}f(s), \qquad s \in S,
\end{equation*}
so $M$ is also strong $L^p$-bounded for $p\in(1,\infty)$ and weak $L^1$-bounded. For $p_0 \in [1,\infty)$ and $f \in L^{p_0}_{\loc}(S)$ we define
\begin{equation*}
  M_{p_0}f(s) := \sup_{B \ni s} \,\ipb{\abs{f}}_{p_0,Q} = M\hab{\abs{f}^{p_0}}(s)^{1/p_0}, \qquad s \in S,
\end{equation*}
which is strong $L^p$-bounded for $p \in(p_0,\infty)$ and weak $L^{p_0}$-bounded. This follows from the boundedness of $M$ by rescaling.

\section{Pointwise \texorpdfstring{$\ell^r$}{lr}-sparse domination}\label{section:main}
In this section we will prove a local version of the sparse domination result in Theorem \ref{theorem:main}, from which we will deduce Theorem \ref{theorem:main} by a covering argument using Lemma \ref{lemma:covering}. This local version will use an abstract localization of the operator $T$, since it depends upon the operator at hand as to the most effective localization. For example in the study of a Calder\'on--Zygmund operator it is convenient to localize the function inserted into $T$, for a maximal operator it is convenient to localize the supremum in the definition of the maximal operator and for a Littlewood--Paley operator it is most suitable to localize the defining integral.

\begin{definition}
  Let $(S,d,\mu)$ be a space of homogeneous type with a dyadic system $\ms{D}$, let $X$ and $Y$ be Banach spaces, $p \in [1,\infty)$ and $\alpha \geq 1$. For a bounded operator
$$T\colon L^{p}(S;X) \to L^{p,\infty}(S;Y)$$ we say that a family of operators $\cbrace{T_Q}_{Q \in \ms{D}}$ from $L^p(S;X)$ to $L^{p,\infty}(Q;Y)$ is an \emph{$\alpha$-localization family of $T$} if for all $Q\in \ms{D}$ and $f \in L^{p}(S;X)$  we have
\begin{align*}
 T_Q(f \ind_{\alpha Q})(s)&=T_Qf(s), & & s \in Q , & & \text{(Localization)}\\
  \nrmb{T_Q(f\ind_{\alpha Q})(s)}_Y &\leq \nrmb{T(f\ind_{\alpha Q})(s)}_Y, & &s \in Q, & & \text{(Domination)}
\end{align*}
For $Q,Q'\in \ms{D}$ with $Q' \subseteq Q$ we  define the difference operator
\begin{align*}
T_{Q\setminus Q'}f(s)&:= T_{Q}f(s) - T_{Q'}f(s), \qquad s \in Q'.
\end{align*}
 and for $Q\in \ms{D}$ the \emph{localized sharp grand maximal truncation operator}
 \begin{align*}
  \mc{M}_{T,Q}^{\#}&f(s) :=\sup_{\substack{Q'\in \ms{D}(Q):\\s \in Q'}}\, \esssup_{s',s'' \in Q'} \,\nrmb{(T_{Q\setminus Q'}) f(s')-(T_{Q\setminus Q'}) f(s'')}_Y, \qquad s \in S.
\end{align*}
\end{definition}

In order to obtain interesting results, one needs to be able to recover the boundedness of $T$ from the boundedness of  $T_Q$ uniformly in $Q \in \ms{D}$.  The canonical example of an $\alpha$-localization family is
\begin{align*}
T_Qf(s) &:=T(f \ind_{\alpha Q})(s), \qquad s \in Q.
\end{align*}
for all $Q \in \ms{D}$ and it is exactly this choice that will lead to Theorem \ref{theorem:main}.
We are now ready to prove our main result, which is a local, more general version of Theorem \ref{theorem:main}.

\begin{theorem}\label{theorem:localsparse}
Let $(S,d,\mu)$ be a space of homogeneous type with dyadic system $\ms{D}$ and let  $X$ and $Y$ be Banach spaces. Take $p_1,p_2,r \in [1,\infty)$, set $p_0:=\max\cbrace{p_1, p_2}$ and take $\alpha \geq 1$. Suppose that
\begin{itemize}
  \item $T$ is a bounded operator from $L^{p_1}(S;X)$ to $L^{p_1,\infty}(S;Y)$ with $\alpha$-localization family $\cbrace{T_Q}_{Q \in \ms{D}}$.
  \item $\mc{M}_{T,Q}^{\#}$ is bounded from $L^{p_2}(S;X)$ to $L^{p_2,\infty}(S)$ uniformly in $Q \in \ms{D}$.
  \item For all $Q_1,\ldots,Q_n \in \ms{D}$ with $Q_n\subseteq \ldots\subseteq Q_1$ and any $f \in L^p(S;X)$
\begin{equation*}
  \qquad \nrmb{T_{Q_1}f(s)}_Y\leq C_r \has{\nrmb{T_{Q_n}{f}(s)}_Y^r+\sum_{k=1}^{n-1}\nrmb{T_{Q_{k}\setminus Q_{k+1}}f(s)}_Y^r}^{1/r},\quad s \in Q_n.
\end{equation*}
\end{itemize}
 Then for any $f \in L^{p_0}(S;X)$ and $Q\in \ms{D}$  there exists a $\frac12$-sparse collection of dyadic cubes $\mc{S}\subseteq \ms{D}(Q)$ such that
\begin{align*}
  \nrmb{T_Qf(s)}_Y\lesssim_{S,\ms{D},\alpha} C_T\,C_r \,
  \has{ \sum_{P \in \mc{S}} \ipb{\nrm{f}_X}_{p_0,\alpha P}^r \ind_P(s)}^{1/r},\qquad s \in Q,
  \end{align*}
with $C_T:={\nrm{T}_{L^{p_1}\to L^{p_1,\infty}} + \sup_{P \in \ms{D} } \nrm{\mc{M}_{T,P}^{\#}}_{L^{p_2}\to L^{p_2,\infty}}}.$
\end{theorem}

The assumption in the third bullet in Theorem \ref{theorem:localsparse} replaces the $r$-sub\-linearity assumption in Theorem \ref{theorem:main}. We will call this assumption a \emph{localized $\ell^r$-estimate}.

\begin{proof}
Fix $f \in L^p(S,X)$ and $Q \in \ms{D}$. We will prove the theorem in two steps: we will first construct the $\frac12$-sparse family of cubes $\mc{S}$ and then show that the sparse expression associated to $\mc{S}$ dominates $T_Qf$ pointwise.

\textbf{Step 1:}
We will construct the $\frac12$-sparse family of cubes $\mc{S}$ iteratively. Given a collection of pairwise disjoint cubes $\mc{S}^k$ for some $k \in \N$ we will first describe how to construct $\mc{S}^{k+1}$. Afterwards we can inductively define $\mc{S}^k$ for all $k \in \N$ starting from $\mc{S}^1 = \cbrace{Q}$ and set $\mc{S}:=\bigcup_{k \in \N} \mc{S}^k$.

Fix a $P\in \mc{S}^k$ and for $\lambda\geq1$ to be chosen later define
\begin{align*}
  \Omega_{P}^1&:= \cbraces{s \in P: \nrm{T_{P} f(s)}_Y> \lambda\, C_T \,\ipb{\nrm{f}_X}_{p_0,\alpha P}}\\
  \Omega_{P}^2&:= \cbraces{s \in P: \mc{M}_{T,P}^{\#}(f)(s)> \lambda \,C_T\,\ipb{\nrm{f}_X}_{p_0,\alpha P}}
\end{align*}
and $\Omega_P:= \Omega_{P}^1 \cup  \Omega_{P}^2$. Let $c_1\geq 1$, depending on $S$, $\ms{D}$ and $\alpha$, be such that $\mu(\alpha P) \leq c_1\,\mu(P)$. By the domination property of the $\alpha$-localization family we have
\begin{align*}
\nrm{T_P f(s)}_Y &\leq \nrm{T(f\ind_{\alpha P})(s)}_Y, & &s \in P,
\intertext{and by the localization property}
  \mc{M}_{T,P}^{\#}(f)(s)&=\mc{M}_{T,P}^{\#}(f\ind_{\alpha P})(s),  & &s \in P.
\end{align*}
  Thus by the weak boundedness assumptions on $T$ and $\mc{M}^{\#}_{T,P}$ and  H\"older's inequality we have for $i=1,2$
\begin{align}\label{eq:holdercomp}
    \mu(\Omega_{P}^i) &\leq \has{\frac{\nrm{f \ind_{\alpha P}}_{L^{p_i}(S;X)}}{\lambda \,\ipb{\nrm{f}_X}_{p_0,\alpha P}}}^{p_i}
    = \frac{\ipb{\nrm{f}_X}_{p_i,\alpha P}^{p_i}}{\lambda^{p_i}\ipb{\nrm{f}_X}_{p_0,\alpha P}^{p_i}} \mu(\alpha P)
     \leq  \frac{ c_1}{\lambda} \, \mu(P).
\end{align}
Therefore it follows that
\begin{equation}\label{eq:OmegatoQ}
  \mu(\Omega_P) \leq \frac{2c_1}{\lambda}  \mu(P).
\end{equation}
To construct the cubes in $\mc{S}^{k+1}$ we will use a local Calder\'on--Zygmund decomposition (see e.g. \cite[Lemma 4.5]{FN18}) on
\begin{equation*}
  \Omega_{P,\rho}:= \cbrace{s \in P: M^{\ms{D}(P)} (\ind_{\Omega_P})>\tfrac1\rho}, \qquad \rho>0
\end{equation*}
 which will be a proper subset of $P$ for our choice of $\lambda$ and $\rho$. Here $M^{\ms{D}(P)}$ is the dyadic Hardy--Littlewood maximal operator with respect to the restricted dyadic system $\ms{D}(P)$.
 The local Calder\'on--Zygmund decomposition yields a pairwise disjoint collection of cubes $\mc{S}_P \subseteq \ms{D}(P)$ and a constant $c_2 \geq 2$, depending on $S$ and $\ms{D}$, such that $\Omega_{P,c_2}= \textstyle{\bigcup_{P'\in \mc{S}_P}}P'$ and
\begin{equation}\label{eq:Pjsize}
 \tfrac{1}{c_2} \, \mu(P') \leq \mu(P'\cap \Omega_{P}) \leq \tfrac{1}{2} \,\mu(P'),\qquad P'\in \mc{S}_P.
\end{equation}
Then by \eqref{eq:OmegatoQ}, \eqref{eq:Pjsize} and the disjointness of the cubes in $\mc{S}_P$ we have
\begin{align*}
     \sum_{P'\in \mc{S}_P} \mu(P') \leq c_2 \, \sum_{P'\in \mc{S}_P}  \mu(P'\cap \Omega_P) \leq c_2 \,\mu(\Omega_P) \leq \frac{2c_1c_2}{\lambda}  \mu(P).
\end{align*}
 Therefore, by choosing $\lambda=4c_1c_2$, we have $\sum_{P'\in \mc{S}_P} \mu(P')  \leq \frac12 \mu(P)$.  This choice of $\lambda$ also ensures that $\Omega_{P,c_2}$ is a proper subset of $P$ by as claimed before. We define $S^{k+1} := \bigcup_{P \in \mc{S}^k} \mc{S}_P$.

Now take $\mc{S}^1 = \cbrace{Q}$, iteratively define $\mc{S}^k$ for all $k \in \N$  as described above and set $\mc{S} :=\bigcup_{k \in \N} \mc{S}^k$. Then $\mc{S}$ is $\frac12$-sparse family of cubes, since for any $P\in \mc{S}$ we can set $$E_P:= P\setminus \bigcup_{P'\in \mc{S}_P} P',$$ which are pairwise disjoint by the fact that $\bigcup_{P'\in \mc{S}^{k+1}} P' \subseteq \bigcup_{P\in \mc{S}^k} P$ for all $k \in \N$ and we have
\begin{equation*}
  \mu(E_P) = \mu(P) - \sum_{P'\in \mc{S}_P} \mu(P') \geq \frac12 \mu(P).
\end{equation*}

\textbf{Step 2:}
 We will now check that the sparse expression corresponding to $\mc{S}$ constructed in Step 1 dominates $T_Qf$ pointwise. Since
\begin{equation*}
\lim_{k \to \infty}  \mu\hab{\bigcup_{P\in \mc{S}^k}P} \leq \lim_{k \to \infty} \frac{1}{2^k}\, \mu(Q) =0,
\end{equation*}
 we know that there is a set $N_0$ of measure zero  such that for all $s \in Q\setminus N_0$ there are only finitely many  $k \in \N$ with $s \in  \bigcup_{P\in \mc{S}^k}P$. Moreover by the Lebesgue differentiation theorem we have for any $P \in \mc{S}$ that
  $\ind_{\Omega_P} (s) \leq M^{\ms{D}(P)}(\ind_{\Omega_P})(s)$
for a.e. $s \in P$. Thus
\begin{equation}\label{eq:Npprop}
  \Omega_P\setminus N_P \subseteq  \Omega_{P,1}\subseteq \Omega_{P,c_2} = \bigcup_{P'\in \mc{S}_P} P'
\end{equation} for some set $N_P$ of measure zero. We define
 $N:= N_0 \cup\bigcup_{P \in \mc{S}}N_P,$  which is a set of measure zero.

  Fix $s \in Q\setminus N$ and take the largest $n \in \N$ such that $s \in \bigcup_{P\in \mc{S}^n}P$, which exists since $s \notin N_0$. For $k =1,\ldots,n$ let $P_k \in \mc{S}^k$ be the unique cube such that $s \in P_k$ and note that by construction we have
  $P_{n} \subseteq  \ldots \subseteq P_1=Q.$
Using the localized $\ell^r$-estimate of $T$ we split $\nrm{T_Qf(s)}_Y^r $ into two parts
\begin{align*}
  \nrmb{T_Qf(s)}_Y^r  &\leq C_r^r \has{\nrmb{T_{P_{n}}f(s)}_Y^r + \sum_{k=1}^{n-1}\nrmb{T_{P_k\setminus P_{k+1}}f(s)}_Y^r}\\
  &=:C_{r}^r \has{ \hspace{2pt} \text{\framebox[15pt]{A}}+  \text{\framebox[15pt]{B}} \hspace{2pt} }.
\end{align*}

For \framebox[15pt]{A} note that $s \notin N_{P_n}$ and $s \notin \bigcup_{P'\in \mc{S}^{n+1}}P' $ and therefore by \eqref{eq:Npprop} we know that $s \in P_n \setminus \Omega_{P_n}$. So by the definition of $\Omega_{P_n}^1$
\begin{equation*}
  \text{\framebox[15pt]{A}} \leq \lambda^r \, C_T^r\, \ipb{\nrm{f}_X}_{p_0,\alpha P_n}^r.
\end{equation*}
For $1\leq k \leq n-1$ we have by \eqref{eq:OmegatoQ} and \eqref{eq:Pjsize} that
\begin{align}\label{eq:Pkmorethan14}
\begin{aligned}
  \mu\hab{P_{k+1} \setminus (\Omega_{P_{k+1}} \cup \Omega_{P_{k}}) }&\geq \mu(P_{k+1}) - \mu(\Omega_{P_{k+1}}) - \mu(P_{k+1} \cap \Omega_{P_k})\\
  &\geq \mu(P_{k+1}) - \frac{1}{2c_2}\mu({P_{k+1}}) - \frac{1}{2}\mu(P_{k+1})>0,
  \end{aligned}
\end{align}
 so $P_{k+1} \setminus (\Omega_{P_{k+1}} \cup \Omega_{P_k})$ is non-empty. Take $s' \in P_{k+1} \setminus (\Omega_{P_{k+1}} \cup \Omega_{P_k})$, then we have

\begin{equation*}
\begin{aligned}
  \nrmb{T_{P_{k}\setminus P_{k+1}}f(s)}_Y &\leq \nrmb{T_{P_{k}\setminus P_{k+1}}f(s)-T_{P_{k}\setminus P_{k+1}}f(s')}_Y + \nrmb{T_{P_{k}\setminus P_{k+1}} f(s')}_Y\\
&\leq \mc{M}_{T,P_k}^{\#}f(s')  + \nrmb{T_{P_k}(s')}_Y+\nrmb{T_{P_{k+1}}(s')}_Y\\
&\leq 2\lambda \,C_T\,\hab{\ipb{\nrm{f}_X}_{p_0,\alpha P_{k}} + \ipb{\nrm{f}_X}_{p_0,\alpha P_{k+1}}},
\end{aligned}
\end{equation*}
where we used the definition of $\mc{M}_{T,P_k}^{\#}$ and $T_{P_{k+1}\setminus P_k}$ in the second inequality and $s' \notin \Omega_{P_{k+1}} \cup \Omega_{P_k}$ in the third inequality. Using $(a+b)^r \leq 2^{r-1}(a^r+b^r)$ for any $a,b>0$ this implies that
\begin{align*}
   \text{\framebox[15pt]{B}}  &\leq \sum_{k=1}^{n-1} 2^r 2^{r-1} \lambda^r \,C_T^r\, \has{ \ipb{\nrm{f}_X}^r_{p_0,\alpha P_k} + \ipb{\nrm{f}_X}^r_{p_0,\alpha P_{k+1}}}\\
   &\leq \sum_{k=1}^{n} 4^r \lambda^r \,C_T^r\,  \ipb{\nrm{f}_X}^r_{p_0,\alpha P_k}.
\end{align*}
Combining the estimates for  \framebox[15pt]{A} and $\text{\framebox[15pt]{B}}$ we obtain
\begin{align*}
  \nrmb{T_Qf(s)}_Y
  &\leq 5\,\lambda\, C_T\, C_{r}\, \has{\sum_{k=1}^n\ipb{\nrm{f}_X}_{p_0,\alpha P_k}^r }^{1/r}\\
  &= 5\,\lambda\, C_T\, C_{r}\, \has{\sum_{P \in \mc{S}}\ipb{\nrm{f}_X}_{p_0,\alpha P}^r \ind_{P}(s)}^{1/r}.
\end{align*}
Since $s \in Q\setminus N$ was arbitrary and $N$ has measure zero, this inequality holds for a.e. $s \in Q$.
Noting that $\lambda = 4c_1c_2$ and $c_1$ and $c_2$ only depend on $S$, $\alpha$ and $\ms{D}$  finishes the proof of the theorem.
\end{proof}

As announced Theorem \ref{theorem:main} now follows directly from Theorem \ref{theorem:localsparse} and a covering argument with Lemma \ref{lemma:covering}.

\begin{proof}[Proof of Theorem \ref{theorem:main}.] We will prove Theorem \ref{theorem:main} in three steps: we will first show that the assumptions of Theorem \ref{theorem:main} imply the assumptions of Theorem \ref{theorem:localsparse}, then we will improve the local conclusion of Theorem \ref{theorem:localsparse} to a global one and finally we will replace the averages over the dilation $\alpha P$ in the conclusion of Theorem \ref{theorem:localsparse} by the average over larger cubes $P'$.

To start let $\ms{D}^1,\ldots,\ms{D}^m$ be as in Proposition \ref{proposition:dyadicsystem} with parameters $c_0$, $C_0$, $\delta$ and $\gamma$, which only depend on $S$.

\textbf{Step 1:}  For any $Q \in \ms{D}^1$ define
$T_Q$
 by $T_Qf(s) := T(f\ind_{\alpha Q})(s)$ for $s \in Q$. Then:
 \begin{itemize}
   \item $\cbrace{T_Q}_{Q \in \ms{D}^1}$ is an $\alpha$-localization family of $T$.
   \item For any $Q \in \ms{D}^1$ and $f \in L^{p_1}(S;X)$ we have
\begin{align*}
  \mc{M}^{\#}_{T,Q}f(s)  &\leq\mc{M}^{\#}_{T,\alpha}(f\ind_{\alpha Q})(s), \qquad s \in Q.
\end{align*}
 So by the weak $L^{p_2}$-boundedness of $\mc{M}^{\#}_{T,\alpha}$ it follows that $\mc{M}^{\#}_{T,Q}f$ is weak $L^{p_2}$-bounded uniformly in $Q \in \ms{D}^1$.
   \item For any  $f \in L^p(S;X)$ and $Q_1,\ldots,Q_n \in \ms{D}^1$ with $Q_n \subseteq \ldots \subseteq Q_1$ the functions $f_k:= f\ind_{\alpha Q_k \setminus \alpha Q_{k+1}}$ for $k=1,\ldots,n-1$ and $f_n:= f\ind_{\alpha Q_n}$ are disjointly supported. Thus by the $r$-sublinearity of $T$
       \begin{equation*}
\nrmb{T_{Q_1}f(s)}_Y\leq C_r \has{\nrmb{T_{Q_n}{f}(s)}_Y^r+\sum_{k=1}^{n-1}\nrmb{T_{Q_{k}\setminus Q_{k+1}}f(s)}_Y^r}^{1/r}, \qquad s \in Q_n.
\end{equation*}
 \end{itemize}
So the assumptions of Theorem \ref{theorem:localsparse} follow from the assumptions of Theorem \ref{theorem:main}.

\textbf{Step 2:} Let $f \in L^p(S;X)$ be boundedly supported. First suppose that $\diam(S) = \infty$ and let $E$ be a ball containing the support of $f$. By Lemma \ref{lemma:covering} there is a partition $\mc{D} \subseteq \ms{D}^1$ such that $E \subseteq \alpha Q$ for all $Q \in \mc{D}$. Thus by Theorem \ref{theorem:localsparse} we can
find a $\frac{1}{2}$-sparse collection of cubes $\mc{S}_Q \subseteq \ms{D}^1(Q)$ for every $Q \in \mc{D}$ with
\begin{align}
 \notag \nrmb{Tf(s)}_Y &\lesssim_{S,\alpha}  C_T\, C_r\, \has{\sum_{P \in \mc{S}_Q} \ipb{\nrm{f}_X}_{p_0,\alpha P}^r\ind_P(s)}^{1/r}, \qquad s \in Q,
\intertext{where we used that $T_Qf = T(f \ind_{\alpha Q}) =Tf$ as $\supp f \subseteq \alpha Q$.
Since $\mc{D}$ is a partition, $\mc{S} := \bigcup_{Q \in \mc{D}}S_Q$ is also a $\frac{1}{2}$-sparse collection of cubes with}
 \label{eq:sparsedomdilation}  \nrmb{Tf(s)}_Y &\lesssim_{S,\alpha}  C_T\, C_r\, \has{\sum_{P \in \mc{S}} \ipb{\nrm{f}_X}_{p_0,\alpha P}^r\ind_P(s)}^{1/r}, \qquad s \in S,
\end{align}
If $\diam (S) < \infty$, then \eqref{eq:sparsedomdilation}  follows directly from Theorem \ref{theorem:localsparse} since $S \in \ms{D}$ in that case.

\textbf{Step 3:} For any $P \in \mc{S}$ with center $z$ and sidelength $\delta^k$ we can find a $P' \in \ms{D}^j$ for some $1\leq j\leq m$ such that
\begin{equation*}
\alpha P = B(z,\alpha  C_0 \cdot \delta^k) \subseteq P', \qquad \diam(P') \leq \gamma \alpha  C_0 \cdot\delta^k.
\end{equation*}
Therefore there is a $c_1>0$ depending on $S$ and  $\alpha$ such that
$$\mu(P') \leq \mu \hab{B(z,\gamma \alpha  C_0 \cdot \delta^k)} \leq c_1\, \mu \hab{ B(z,c_0\cdot \delta^k) } \leq c_1\, \mu(P).$$
So by defining $E_{P'}:= E_P$ we can conclude that the collection of cubes $\mc{S}' := \cbrace{P':P \in \mc{S}}$ is $\frac{1}{2c_1}$-sparse. Moreover since $\alpha P \subseteq P'$ and $\mu(P') \leq c_1 \, \mu(P) \leq c_1\,\mu(\alpha P)$ for any $P \in \mc{S}$, we have
\begin{equation*}
  \ipb{\nrm{f}_X}_{p_0,\alpha P} \leq c_1 \ipb{\nrm{f}_X}_{p_0,P'}.
\end{equation*}
Combined with \eqref{eq:sparsedomdilation} this proves the sparse domination in the conclusion of  Theorem \ref{theorem:main}.
\end{proof}

\begin{remark}~\label{remark:mainpointwise}
The assumption $\alpha \geq {3c_d^2/\delta}$ in Theorem \ref{theorem:main} arises from the use of  Lemma \ref{lemma:covering}, which transfers  the local sparse domination estimate of Theorem \ref{theorem:localsparse} to the global statement of Theorem \ref{theorem:main}. To deduce weighted estimates the local sparse domination estimate of Theorem \ref{theorem:localsparse} suffices by testing against boundedly supported functions. However the operator norm of $\mc{M}_{T,\alpha}^{\#}$ usually  becomes easier to estimate for larger $\alpha$, so the lower bound on $\alpha$ is not restrictive.
\end{remark}

\subsection*{Further generalizations}
Our main theorems,  Theorem \ref{theorem:main} and Theorem \ref{theorem:localsparse}, allow for various further generalizations. One can for instance change the boundedness assumptions on $T$ and $\mc{M}^{\#}_{T,\alpha}$, treat multilinear operators, or deduce domination by sparse forms for operators that do not admit a pointwise sparse estimate. We end this section by sketching some of these possible generalizations.

In \cite[Section 3]{LO19} various variations and extensions of the main result in \cite{LO19} are outlined. In particular they show:
\begin{itemize}
  \item The sparse domination for an individual function follows from assumptions on the same function. This can be exploited to prove a sparse $T(1)$-type theorem, see \cite[Section 4]{LO19}.
  \item One can use certain Orlicz estimates to deduce sparse domination with Orlicz averages.
  \item The method of proof extends to the multilinear setting (see also \cite{Li18}).
\end{itemize}
 Our results can also be extended in these directions, which we leave to the interested reader. In the remainder of this section, we will explore some further directions in which our results can be extended.

 \bigskip

Sparse domination techniques have been successfully applied to \emph{fractional integral operators}, see e.g. \cite{CB13, CB13b,Cr17,IRV18}. In these works sparse domination and sharp weighted estimates are deduced for e.g.
 the Riesz potentials, which for $0<\alpha<d$ and a Schwartz function $f\colon \R^d \to \C$ are given by
\begin{equation*}
  I_\alpha f(s):= \int_{\R^d}\frac{f(t)}{\abs{s-t}^{d-\alpha}}\dd t, \qquad s \in \R^d,
\end{equation*}
A key feature of such operators is that they are not (weakly) $L^p$-bounded, but bounded from $L^p(\R^d)$ to $L^q(\R^d)$, where $p,q \in (1,\infty)$ are such that $\frac1p=\frac1q+\frac\alpha{d}$. The sparse domination that one obtains in this case involves fractional sparse operators, in which the usual averages $\ip{\abs{f}}_{p,Q}$ are replaced by fractional averages.

These operators  fit in our framework with minimal effort. Indeed, upon inspection of the proof of Theorem \ref{theorem:localsparse} it becomes clear that the only place where we use the boundedness of $T$ and $\mc{M}^{\#}_{T,\alpha}$ is in \eqref{eq:holdercomp}. Replacing the bounds with the off-diagonal bounds arising from fractional integral operators, we obtain the following variant of Theorem \ref{theorem:main}.

\begin{theorem}\label{theorem:fractional}
Let $(S,d,\mu)$ be a space of homogeneous type and let $X$ and $Y$ be Banach spaces. Take $p_0,q_0,r \in [1,\infty)$. Take $\alpha \geq 3c_d^2/\delta$, where $c_d$ is the quasi-metric constant and $\delta$ is as in Proposition \ref{proposition:dyadicsystem}. Assume the following conditions:
\begin{itemize}
  \item $T$ is a bounded linear operator from $L^{p_0}(S;X)$ to $L^{q_0,\infty}(S;Y)$.
  \item $\mc{M}_{T,\alpha}^{\#}$ is a bounded operator from $L^{p_0}(S;X)$ to $L^{q_0,\infty}(S)$.
  \item $T$ is $r$-sublinear.
\end{itemize}
 Then there is an $\eta \in (0,1)$ such that for any boundedly supported $f \in L^{p_0}(S;X)$ there is an $\eta$-sparse collection of cubes $\mc{S}$ such that
\begin{align*}
   \nrm{ Tf(s)}_Y&\lesssim_{S,\alpha} C_T \,C_r\, \has{\sum_{Q \in \mc{S}} \mu(\alpha P)^{\frac{r}{p_0}-\frac{r}{q_0}}\ipb{\nrm{f}_X}_{p_0,Q}^r \ind_{Q}(s)}^{1/r}, \qquad s \in S,
\end{align*}
where $C_T={\nrm{T}_{L^{p_0}\to L^{p_0,\infty}} + \nrm{\mc{M}_{T,\alpha}^{\#}}_{L^{p_0}\to L^{p_0,\infty}}}$ and $C_r$ is the $r$-sublinearity constant.
\end{theorem}

\begin{proof}
The proof is the same as the proof of Theorem \ref{theorem:main}, using an adapted version of Theorem \ref{theorem:localsparse} with the canonical $\alpha$-localization family
$$T_Qf(s) = T(\ind_{\alpha Q}f)(s), \qquad s \in Q.$$
  The only thing that changes in the proof of Theorem \ref{theorem:localsparse} is the definition of $\Omega_P^1$ and $\Omega_P^2$ and the computation in \eqref{eq:OmegatoQ}. Indeed, we define
\begin{align*}
  \Omega_{P}^1&:= \cbraces{s \in P: \nrm{T_{P} f(s)}_Y> \lambda\, C_T \,\mu(\alpha P)^{\frac1{p_0}-\frac1{q_0}}\ipb{\nrm{f}_X}_{p_0,\alpha P}}\\
  \Omega_{P}^2&:= \cbraces{s \in P: \mc{M}_{T,P}^{\#}(f)(s)> \lambda \,C_T\,\mu(\alpha P)^{\frac1{p_0}-\frac1{q_0}}\ipb{\nrm{f}_X}_{p,\alpha P}}
\end{align*}
and then by the assumptions on $T$ and $\mc{M}^{\#}_{T,P}$  we have for $i=1,2$
\begin{align*}
    \mu(\Omega_{P}^i) &\leq \has{\frac{\nrm{f \ind_{\alpha P}}_{L^{p_0}(S;X)}}{\lambda \,\mu(\alpha P)^{\frac1{p_0}-\frac1{q_0}}\ipb{\nrm{f}_X}_{p_0,\alpha P}}}^{q_0}
    = \frac{\ipb{\nrm{f}_X}_{p_0,\alpha P}^{q_0}}{\lambda^{q_0}\ipb{\nrm{f}_X}_{p_0,\alpha P}^{q_0}} \mu(\alpha P)
     \leq  \frac{ c_1}{\lambda} \, \mu(P).
\end{align*}
which proves \eqref{eq:OmegatoQ}. In Step 2 of the proof of Theorem \ref{theorem:localsparse} one needs to keep track of the factor $\mu(\alpha P)^{\frac1{p_0}-\frac1{q_0}}$ in the estimates.
\end{proof}

In the celebrated paper \cite{BFP16} by Bernic\'ot, Frey and Petermichl, domination by \emph{sparse forms} was introduced to treat operators falling outside the scope of Calder\'on--Zygmund theory. This method was later adopted by Lerner in \cite{Le19} into his framework to prove sparse domination for rough homogeneous singular integral operators. As our methods are based on Lerner's sparse domination framework, our main result can also be generalized to the sparse form domination  setting.

Let $(S,d,\mu)$ be a space of homogeneous type with a dyadic system $\ms{D}$, let $X$ and $Y$ be Banach spaces, $q \in (1,\infty)$, $p \in [1,q)$ and $\alpha \geq 1$. For a bounded operator
$$T\colon L^{p}(S;X) \to L^{p,\infty}(S;Y)$$ with an $\alpha$-localization family $\cbrace{T_Q}_{Q \in \ms{D}}$ we define the \emph{localized sharp grand $q$-maximal truncation operator} for $Q \in \ms{D}$ by
 \begin{align*}
  &\mc{M}_{T,Q,q}^{\#}f(s):= \\&\hspace{1cm}\sup_{\substack{Q'\in \ms{D}(Q):\\s \in Q'}}\, \has{\avint_{Q'}\avint_{Q'} \,\nrmb{(T_{Q\setminus Q'}) f(s')-(T_{Q\setminus Q'}) f(s'')}_Y^q\dd \mu(s')\dd \mu(s'')}^{1/q}.
\end{align*}
Note that for $q=\infty$ one formally recovers the operator $\mc{M}_{T,Q}^{\#}$.

We will prove a version of Theorem \ref{theorem:localsparse} for operators for which the truncation operators $\mc{M}_{T,Q,q}^{\#}$ are bounded uniformly in $Q \in \ms{D}$ using sparse forms. Of course taking
 \begin{align*}
T_Qf(s) &:=T(f \ind_{\alpha Q})(s), \qquad s \in Q.
\end{align*}
for $Q \in \ms{D}$ as the $\alpha$-localization family one can easily deduce a statement like Theorem \ref{theorem:main} in this setting, which we leave to the interested reader.

\begin{theorem}\label{theorem:sparseform}
Let $(S,d,\mu)$ be a space of homogeneous type with dyadic system $\ms{D}$ and let  $X$ and $Y$ be Banach spaces. Take $q_0 \in (1,\infty]$, $r \in (0,q_0)$, $p_1,p_2 \in [1,q_0)$, set $p_0:=\max\cbrace{p_1, p_2}$ and take $\alpha \geq 1$. Suppose that
\begin{itemize}
  \item $T$ is a bounded operator from $L^{p_1}(S;X)$ to $L^{p_1,\infty}(S;Y)$ with an $\alpha$-localization family $\cbrace{T_Q}_{Q \in \ms{D}}$.
  \item $\mc{M}_{T,Q,q_0}^{\#}$ is bounded from $L^{p_2}(S;X)$ to $L^{p_2,\infty}(S)$ uniformly in $Q \in \ms{D}$.
  \item $T$ satisfies a localized $\ell^r$-estimate.
\end{itemize}
 Then for any $f \in L^{p_0}(S;X)$, $g \in L^{\ha{\frac1r-\frac1{q_0}}^{-1}}(S)$ and $Q\in \ms{D}$  there exists a $\frac12$-sparse collection of dyadic cubes $\mc{S}\subseteq \ms{D}(Q)$ such that
 \begin{equation*}
   \has{\int_Q \nrmb{T_Qf}_Y^r\cdot \abs{g}^r\dd\mu }^{1/r} \lesssim_{S,\ms{D},\alpha,r} C_T \, C_r \has{\sum_{P \in \mc{S}} \mu(P) \ipb{\nrm{f}_X}_{p_0,\alpha P}^r \ipb{\abs{g}}_{\frac{1}{\frac1r-\frac1{q_0}}, P}^r}^{1/r}
 \end{equation*}
with $C_T:={\nrm{T}_{L^{p_1}\to L^{p_1,\infty}} + \sup_{P \in \ms{D} } \nrm{\mc{M}_{T,P,q_0}^{\#}}_{L^{p_2}\to L^{p_2,\infty}}}$ and $C_r$ the constant from the localized $\ell^r$-estimate.
\end{theorem}

\begin{proof}
We construct the sparse collection of cubes $\mc{S}$ exactly as in Step 1 of the proof of Theorem \ref{theorem:localsparse}, using $\mc{M}_{T,P,q_0}^{\#}$ instead of $\mc{M}_{T,P}^{\#}$ in the definition of $\Omega_P^2$. We will check that sparse form corresponding to $\mc{S}$ satisfies the claimed domination property, which will roughly follow the same lines as Step 2 of the proof of Theorem \ref{theorem:localsparse}.

Fix $f \in L^{p_0}(S;X)$ and $g \in L^{\ha{\frac1r-\frac1{q_0}}^{-1}}(S)$. Note that for a.e. $s \in Q$
there are only finitely many  $k \in \N$ with $s \in  \bigcup_{P\in \mc{S}^k}P$. So we can use the localized $\ell^r$-estimate of $T$ to split
\begin{equation}\label{eq:splitform}\begin{aligned}
  \int_Q \nrmb{T_Qf}_Y^r\cdot \abs{g}^r  &\leq C_r^r \sum_{k\in \N} \sum_{P \in \mc{S}^k} \has{\int_{P \setminus \bigcup_{P' \in \mc{S}^{k+1}}P'} \nrmb{T_Pf}_Y^r\cdot \abs{g}^r \\&\hspace{2cm}+\sum_{P' \in \mc{S}^{k+1}:P'\subseteq P} \int_{P'}\nrmb{T_{P\setminus P'}f}_Y^r\cdot \abs{g}^r}\\
  &=:C_{r}^r \sum_{k\in \N} \sum_{P \in \mc{S}^k} \has{ \hspace{2pt} \text{\framebox[20pt]{A$_{P}$}}+  \text{\framebox[20pt]{B$_{P}$}} \hspace{2pt} }.\end{aligned}
\end{equation}
Fix $k \in \N$ and $P \in \mc{S}^k$. As in the estimate for \framebox[15pt]{A} in Step 2 of the proof of Theorem \ref{theorem:localsparse}, we have
\begin{align*}
  \text{\framebox[20pt]{A$_{P}$}} &\leq \lambda^r \, C_T^r\, \ipb{\nrm{f}_X}_{p_0,\alpha P}^r \int_P\abs{g}^r \leq  \lambda^r \, C_T^r\, \mu(P) \ipb{\nrm{f}_X}_{p_0,\alpha P}^r \ip{\abs{g}}_{\frac{1}{\frac1r-\frac1{q_0}}}^r,
\end{align*}
using H\"older's inequality in the second inequality.
For $P' \in \mc{S}^{k+1}$ such that $P'\subseteq P$ we have as in \eqref{eq:Pkmorethan14} that
\begin{align*}
  \mu\hab{P' \setminus (\Omega_{P'} \cup \Omega_{P}) }  &\geq \frac14 \mu(P').
\end{align*}
Therefore we can estimate each of the terms in the sum in \framebox[20pt]{B$_{P}$} as follows
\begin{equation*}
\begin{aligned}
  \int_{P'}&\nrmb{T_{P\setminus P'}f}_Y^r\cdot \abs{g}^r\\ &\leq 2^{r}\int_{P'} \avint_{P' \setminus (\Omega_{P} \cup \Omega_{P'})} \nrmb{T_{P\setminus P'}f(s)-T_{P\setminus P'}f(s')}_Y^r\cdot \abs{g(s)}^r \dd \mu(s')\dd \mu(s)\\
  &\hspace{2.2cm}+ 2^{r}\int_{P'} \avint_{P' \setminus (\Omega_{P} \cup \Omega_{P'})} \nrmb{T_{P\setminus P'}f(s')}_Y^r \cdot  \abs{g(s)}^r \dd \mu(s')\dd \mu(s)\\
  &\leq 2^{r+2} \mu(P') \inf_{s'' \in P'} \mc{M}_{T,P,{q_0}}^{\#}f(s'')^r\cdot \ip{\abs{g}}_{\frac{1}{\frac1r-\frac1{q_0}}, P'}^r\\
  &\hspace{1.5cm}+ 2^{2r} \mu(P') \avint_{P' \setminus (\Omega_{P} \cup \Omega_{P'})} \nrmb{T_{P}f}_Y^r +\nrmb{T_{P'}f}_Y^r \dd \mu \cdot \ip{\abs{g}}_{r, P'}^r\\
  &\leq  4^{r+2}\lambda^r C_T^r\,\mu(P') \hab{\ipb{\nrm{f}_X}_{p_0,\alpha P}^r + \ipb{\nrm{f}_X}_{p_0,\alpha P'}^r} \ip{\abs{g}}_{\frac{1}{\frac1r-\frac1{q_0}}, P'}^r
\end{aligned}
\end{equation*}
where we used
H\"older's inequality and the definitions of $\mc{M}_{T,P,{q_0}}^{\#}$ and $T_{P\setminus P'}$ in the second inequality and the definitions of $\Omega_{P}$ and $\Omega_{P'}$ in the third inequality. Furthermore we note that by H\"olders inequality we have
\begin{align*}
  \sum_{\substack{P' \in \mc{S}^{k+1}:\\P'\subseteq P}} \mu(P') \, \ip{\abs{g}}_{\frac{1}{\frac1r-\frac1{q_0}}, P'}^r &\leq \has{\sum_{\substack{P' \in \mc{S}^{k+1}:\\P'\subseteq P}} \int_{P'}\abs{g}^{\frac{1}{\frac1r-\frac1{q_0}}}\dd \mu}^{1-\frac{r}{{q_0}}} \cdot \has{\sum_{\substack{P' \in \mc{S}^{k+1}:\\P'\subseteq P}} \mu(P')}^{r/{q_0}}\\
  &\leq \has{ \int_{P}\abs{g}^{\frac{1}{\frac1r-\frac1{q_0}}}\dd \mu}^{1-\frac{r}{{q_0}}} \cdot {\mu(P)}^{r/{q_0}}= \mu(P) \ip{\abs{g}}_{\frac{1}{\frac1r-\frac1{q_0}}, P}^r
\end{align*}
Thus for \framebox[20pt]{B$_{P}$} we obtain
\begin{align*}
  \text{\framebox[20pt]{B$_{P}$}}\leq 4^{r+2}\lambda^r C_T^r &\has{ \mu(P) \ipb{\nrm{f}_X}_{p_0,\alpha P}^r  \ip{\abs{g}}_{\frac{1}{\frac1r-\frac1{q_0}}, P}^r \\&+ \sum_{P' \in \mc{S}^{k+1}:P'\subseteq P} \mu(P')  \ipb{\nrm{f}_X}_{p_0,\alpha P'}^r \ip{\abs{g}}_{\frac{1}{\frac1r-\frac1{q_0}}, P'}^r}
\end{align*}
Plugging this estimate and the estimate for \framebox[20pt]{A$_{P}$} into \eqref{eq:splitform} yields
\begin{equation*}
  \int_Q \nrmb{T_Qf}_Y^r\cdot \abs{g}^r \dd \mu \leq  4^{r+3} \lambda^r \, C_T^r \, C_r^r \sum_{P \in \mc{S}} \mu(P) \ipb{\nrm{f}_X}_{p_0,\alpha P}^r \ip{\abs{g}}_{\frac{1}{\frac1r-\frac1{q_0}}, P}^r.
\end{equation*}
Since $\lambda = 4c_1c_2$ and $c_1$ and $c_2$ only depend on $S$, $\alpha$ and $\ms{D}$, this finishes the proof of the theorem.
\end{proof}

\section{Weighted bounds for sparse operators}\label{section:weights}
As discussed in the introduction, one of the main motivations to study sparse domination for an operator is to obtain (sharp) weighted bounds. In this section we will introduce Muckenhoupt weights and state weighted $L^p$-bounds for the sparse operators in the conclusions of Theorem \ref{theorem:main} and Theorem \ref{theorem:localsparse}, which are well-known in the Euclidean setting.

Let $(S,d,\mu)$ be a space of homogeneous type. A \emph{weight} is a locally integrable function $w\colon S \to (0,\infty)$. For $p \in [1,\infty)$, a Banach space $X$ and a weight $w$ the weighted Bochner space $L^p(S,w;X)$ is the space of all strongly measurable $f:S \to X$ such that
\begin{equation*}
  \nrm{f}_{L^p(S,w;X)}:= \has{\int_S\nrm{f(s)}_X^pw\dd \mu}^{1/p}<\infty.
\end{equation*}
For $p\in [1,\infty)$ and a weight $w$ we say that $w$ lies in the \emph{Muckenhoupt class $A_p$} and write $w\in A_p$ if its \emph{$A_p$-characteristic} satisfies
\begin{equation*}
  [w]_{A_p}:= \sup_{B\subseteq S} \ip{w}_{1,B}\ip{w^{-1}}_{\frac{1}{p-1},B}<\infty,
\end{equation*}
where the supremum is taken over all balls $B \subseteq S$ and the second factor is replaced by $\esssup_B w^{-1}$ if $p=1$.
For an introduction to Muckenhoupt weights we refer to \cite[Chapter 7]{Gr14a}.

\bigskip

Let $p_0, r \in [1,\infty)$, $p \in (p_0,\infty)$, $w \in A_{p/p_0}$. We are interested in the boundedness on $L^p(S,w)$ of sparse operators of the form
\begin{equation}\label{eq:sparseop2}
  f\mapsto \has{\sum_{Q \in \mc{S}} \ipb{\abs{f}}_{p_0,Q}^r \ind_{Q}}^{1/r},
\end{equation}
which appear in the conclusions of Theorem \ref{theorem:main} and Theorem \ref{theorem:localsparse}.
In the Euclidean case such bounds are thoroughly studied and most of the arguments extend directly to spaces of homogeneous type. For the convenience of the reader we will give a self-contained proof of the strong weighted $L^p$-boundedness of these sparse operators in spaces of homogeneous type, following the proof of \cite[Lemma 4.5]{Le16}. For further results we refer to:
\begin{itemize}
\item Weak weighted $L^p$-boundedness (including the endpoint $p=p_0$), for the sparse operators in \eqref{eq:sparseop2} can be found \cite{HL18, FN18}.
\item More precise bounds in terms of two-weight $A_p$-$A_\infty$-characteristics  for various special cases of the sparse operators in \eqref{eq:sparseop2} can be found in e.g. \cite{FH18, HL18,HP13,LL16}.
  \item Weighted bounds for the fractional sparse operators in Theorem \ref{theorem:fractional} can be found in \cite{FH18}
  \item Weighted bounds for the sparse forms in Theorem \ref{theorem:sparseform} can be found in \cite{BFP16, FN18}.
\end{itemize}

\begin{proposition}\label{proposition:weights}
  Let $(S,d,\mu)$ be a space of homogeneous type, let $\mc{S}$ be an $\eta$-sparse collection of cubes and take $p_0, r \in [1,\infty)$.
For $p \in (p_0,\infty)$, $w \in A_{p/p_0}$ and $f \in L^p(S,w)$ we have
\begin{align*}
  \nrms{\has{\sum_{Q \in \mc{S}} \ipb{\abs{f}}_{p_0,Q}^r \ind_{Q}}^{1/r}}_{ L^p(S,w)} &\lesssim [w]_{A_{p/p_0}}^{\max\cbraceb{\frac{1}{p-p_0},\frac{1}{r}}} \nrm{f}_{L^p(S,w)},
\end{align*}
where
the implicit constant depends on $S,p_0,p,r$ and $\eta$.
\end{proposition}

\begin{proof}
We first note that by Proposition \ref{proposition:dyadicsystem} we may assume without loss of generality that  $\mc{S} \subseteq \ms{D}$, where $\ms{D}$ is an arbitrary dyadic system in $(S,d,\mu)$. Furthermore if $p -p_0\leq r$ we have $\max\cbraceb{\frac{1}{p-p_0},\frac{1}{r}} = \frac{1}{p-p_0}$. Since $\ell^{p-p_0}\hookrightarrow \ell^r$, the case $p -p_0\leq r$ follows from the case $p -p_0=r$, so without loss of generality we may also assume $p\geq p_0+r$.

For a weight $u$ and a measurable set $E$ we define $u(E):=\int_Eu\dd \mu$ and we denote the dyadic Hardy--Littlewood maximal operator with respect to the measure $u\dd \mu$ by $M^{\ms{D},u}$, which
is bounded on $L^p(S,u)$ for all $p \in (1,\infty)$ by Doob's maximal inequality (see e.g. \cite[Theorem 3.2.2]{HNVW16}). Take $f \in L^p(S,w)$, set $q := (p/r)'=\frac{p}{p-r}$ and take
$$g \in L^{q}(S,w^{1-q}) = \hab{L^{p/r}(S,w)}^*.$$
Then we have by the disjointness of the $E_Q$'s associated to each $Q \in \mc{S}$
\begin{equation}\label{eq:gsparse}
\begin{aligned}
  \sum_{Q \in \mc{S}} w(E_Q) \has{\frac{\mu(Q)}{w(Q)}}^{q} \ipb{\abs{g}}_{1,Q}^{q}
  &\leq \sum_{Q \in \mc{S}}\int_{E_Q} M^{\ms{D},w}(gw^{-1})^{q} w \dd\mu\\
  &\leq \nrmb{M^{\ms{D},w}(gw^{-1})}_{L^{q}(S,w)}^{q}\\&\lesssim_{p,r} \nrm{g}_{L^{q}(S,w^{1-q})}^{q}
\end{aligned}
\end{equation}
and similarly, setting $\sigma:= w^{1-(p/p_0)'}$, we have
\begin{equation}\label{eq:fsparse}
\begin{aligned}
  \sum_{Q \in \mc{S}} \sigma(E_Q) \has{\frac{\mu(Q)}{\sigma(Q)}}^{\frac{p}{p_0}} \ipb{\abs{f}^{p_0}}_{1,Q}^{{p}/{p_0}}
  &\leq \nrmb{M^{\ms{D},\sigma}(\abs{f}^{p_0}\sigma^{-1})}_{L^{p/p_0}(S,\sigma)}^{p/p_0}\\
  &\lesssim_{p,p_0} \nrm{f}_{L^{p}(S,w)}^{p}
\end{aligned}
\end{equation}
using $\sigma \cdot \sigma^{-p_0/p}=w$.
Define the constant
\begin{equation*}
  c_w := \sup_{Q \in \ms{D}} \frac{w(Q)^{1/r}}{w(E_Q)^{\frac1r-\frac1p}} \frac{\sigma(Q)^{1/p_0}}{\sigma(E_Q)^{1/p} } \frac{1}{\mu(Q)^{1/p_0}},
\end{equation*}
Then by H\"olders inequality, \eqref{eq:gsparse} and \eqref{eq:fsparse} we have
\begin{align*}
  \int_S \has{\sum_{Q \in \mc{S}} \ipb{\abs{f}}_{p_0,Q}^r \ind_{Q}}\cdot g \dd \mu  &= \sum_{Q \in \mc{S}} \mu(Q) \ipb{\abs{f}^{p_0}}_{1,Q}^{r/p_0} \ip{\abs{g}}_{1,Q}\\
     &\leq  c_w^r \sum_{Q \in \mc{S}} \has{ \sigma(E_Q)^{r/p} \has{\frac{\mu(Q)}{\sigma(Q)}}^{r/p_0} \ipb{\abs{f}^{p_0}}_{1,Q}^{r/p_0}}\\&\hspace{1cm}\cdot\has{ w(E_Q)^{{1/q}}\frac{\mu(Q)}{w(Q)} \ip{\abs{g}}_{1,Q}}\\
   &\lesssim_{p,p_0,r} c_w^r \nrmb{f}_{L^{p}(S,w)}^{r} \nrm{g}_{L^{q}(S,w^{1-q})}.
\end{align*}
So by duality it remains to show  $c_w \lesssim [w]_{A_{p/p_0}}^{\max\cbraceb{\frac{1}{p-p_0},\frac{1}{r}}}$. Fix a $Q \in \ms{D}$ and  note that by H\"olders's inequality we have
\begin{equation*}
   \mu(Q)^{p/p_0}\leq \eta^{p/p_0} \has{\int_{E_Q}w^{p_0/p} w^{-p_0/p} \dd \mu}^{p/p_0}
 \leq \eta^{p/p_0} \,w(E_Q)\,\sigma(E_Q)^{p/p_0-1}.
\end{equation*}
and thus
\begin{equation*}
  \frac{w(Q)}{w(E_Q)}\has{\frac{\sigma(Q)}{\sigma(E_Q)}}^{{p/p_0}-1} \leq \eta^{p/p_0} \frac{w(Q)}{\mu(Q)}\has{\frac{\sigma(Q)}{\mu(Q)}}^{{p/p_0}-1} \lesssim_S \eta^{p/p_0}[w]_{A_{p/p_0}}.
\end{equation*}
Therefore we can estimate
\begin{align*}
  c_w &= \sup_{Q \in \ms{D}} \bracs{\frac{w(Q)}{\mu(Q)}\has{\frac{\sigma(Q)}{\mu(Q)}}^{\frac{p}{p_0}-1}}^{\frac1p} \cdot \bracs{\has{\frac{w(Q)}{w(E_Q)}}^{\frac1r-\frac1p} \has{\frac{\sigma(Q)}{\sigma(E_Q)}}^{\frac1p}}\\
  &\lesssim_S [w]_{A_{p/p_0}}^{\frac1p} \, \sup_{Q \in \ms{D}} \bracs{\frac{w(Q)}{w(E_Q)}\has{\frac{\sigma(Q)}{\sigma(E_Q)}}^{\frac{p}{p_0}-1}}^{\max \cbraceb{{\frac1r-\frac1p}, \frac1p\frac{p_0}{p-p_0}}}\\
  &\lesssim_{S,\eta}  [w]_{A_{p/p_0}}^{\frac1p+\max \cbraceb{{\frac1r-\frac1p}, \frac1p\frac{p_0}{p-p_0}} } = [w]_{A_{p/p_0}}^{\max\cbraceb{\frac{1}{p-p_0},\frac{1}{r}}},
\end{align*}
which  finishes the proof.
\end{proof}

\section{Banach space geometry and $\mc{R}$-boundedness}\label{section:Banachspace}
Before turning to applications of Theorem \ref{theorem:main} and Theorem \ref{theorem:localsparse} in the subsequent sections, we first need to introduce some geometric properties of a Banach space $X$ and the $\mc{R}$-boundedness of a family of operators.

\subsection{Type and cotype}Let $(\varepsilon_k)_{k=1}^\infty$ be a sequence of independent \emph{Rademacher variables} on $\Omega$, i.e. uniformly distributed random variables taking values in $\cbrace{z \in \K:\abs{z} = 1}$. We say that a Banach space $X$ has (Rademacher) type $p \in [1,2]$ if for any $x_1,\ldots,x_n \in X$ we have
\begin{equation*}
  \nrms{\sum_{k=1}^n\varepsilon_k x_k}_{L^2(\Omega;X)} \lesssim_{X,p} \has{\sum_{k=1}^n \nrm{x_k}_X^p}^{1/p},
\end{equation*}
and say that $X$ has nontrivial type if $X$ has type $p>1$.
We say that $X$ has (Rademacher) cotype $q \in [2,\infty]$ if for any $x_1,\ldots,x_n \in X$ we have
\begin{equation*}
   \has{\sum_{k=1}^n \nrm{x_k}_X^q}^{1/q}\lesssim_{X,q} \nrms{\sum_{k=1}^n\varepsilon_k x_k}_{L^2(\Omega;X)},
\end{equation*}
 and say that $X$ has finite cotype if $X$ has cotype $q<\infty$. See \cite[Chapter 7]{HNVW17} for an introduction to type and cotype.

\subsection{Banach lattices and $p$-convexity and $q$-concavity.}
A Banach lattice is a partially ordered Banach space $X$ such that for $x,y \in X$
\begin{equation*}
  \abs{x} \leq \abs{y} \Rightarrow \nrm{x}_X \leq \nrm{y}_Y.
\end{equation*}
On a Banach lattice there are two properties that are closely related to type and cotype. We say that a Banach lattice is \emph{$p$-convex} with $p \in [1,\infty]$ if for $x_1,\ldots,x_n\in X$
\begin{equation*}
  \nrms{\has{\sum_{k=1}^n\abs{x_k}^p}^{1/p}}_{X} \lesssim_{X,p} \has{\sum_{k=1}^n \nrm{x_k}^p}^{1/p},
\end{equation*}
where the sum on the left-hand side is defined through the Krivine calculus. A Banach lattice is called \emph{$q$-concave} for $q \in [1,\infty]$ if for $x_1,\ldots,x_n\in X$
\begin{equation*}
   \has{\sum_{k=1}^n \nrm{x_k}^q}^{1/q} \lesssim_{X,q} \nrms{\has{\sum_{k=1}^n\abs{x_k}^q}^{1/q}}_{X}.
\end{equation*}
If a  Banach lattice has finite cotype then $p$-convexity implies type $p$. Conversely type $p$ implies $r$-convexity for all $1\leq r< p$. Similar relations hold for cotype $q$ and $q$-concavity. We refer to \cite[Chapter 1]{LT79} for an introduction to Banach lattices, $p$-convexity and $q$-concavity.

\subsection{The \texorpdfstring{$\UMD$}{UMD} property} We say that a Banach space $X$ has the {$\UMD$ property} if the martingale difference sequence of any finite martingale in $L^p(\Omega;X)$ is unconditional for some (equivalently all) $p \in (1,\infty)$. The $\UMD$ property implies reflexivity, nontrivial type and finite cotype.
For an  introduction to the theory of $\UMD$ Banach spaces we refer the reader to \cite[Chapter 4]{HNVW16} and \cite{Pi16}.

\subsection{\texorpdfstring{$\mc{R}$}{R}-Boundedness} Let $X$ and $Y$ be Banach spaces and $\Gamma \subseteq \mc{L}(X,Y)$. We say that $\Gamma$ is $\mc{R}$-bounded if for any $x_1,\ldots,x_n$ and $T_1,\ldots,T_n \in \Gamma$ we have
\begin{equation*}
  \has{\E\nrmb{\sum_{k=1}^n \varepsilon_k T_kx_k}^2}^{1/2} \lesssim \has{\E\nrmb{\sum_{k=1}^n \varepsilon_k x_k}^2}^{1/2},
\end{equation*}
 where $(\varepsilon_k)_{k=1}^\infty$ is a sequence of independent {Rademacher variables}
The least admissible implicit constant is denoted by $\mc{R}(\Gamma)$. $\mc{R}$-boundedness is a strengthening of uniform boundedness and is often a key assumption to prove boundedness of operators on Bochner spaces. We refer to \cite[Chapter 8]{HNVW17} for an introduction to $\mc{R}$-boundedness.

\section{The \texorpdfstring{$A_2$}{A2}-theorem for operator-valued  Calder\'on--Zygmund operators in a space of homogeneous type} \label{section:A2}
The $A_2$-theorem, first proved by Hyt\"onen in \cite{Hy12} as discussed in the introduction, states that a Calder\'on--Zygmund operator is bounded on $L^2(\R^d,w)$ with a bound that depends linearly on the $A_2$-characteristic of $w$. From this sharp weighted bounds for all $p\in (1,\infty)$ can be obtained by sharp Rubio de Francia extrapolation \cite{DGPP05}. Since its first proof by Hyt\"onen, the $A_2$-theorem has been extended in various directions. We mention two of these extensions relevant for the current discussion:
\begin{itemize}
  \item The $A_2$-theorem for Calder\'on--Zygmund operators on a geometric doubling metric space
 was first proven by Nazarov, Reznikov and Volberg \cite{NRV13}, afterwards it was proven on a space of homogeneous type
 by Anderson and  Vagharshakyan \cite{AV14} (see also \cite{An15}) using Lerner's mean oscillation decomposition method. It was further extended to the setting of ball bases by Karagulyan \cite{Ka16}.
  \item The $A_2$-theorem for vector-valued Calder\'on--Zygmund operators with operator-valued kernel was proven by H\"anninen and Hyt\"onen \cite{HH14}, using a suitable adapted version of Lerner's median oscillation decomposition.
\end{itemize}
In this section we will prove sparse domination for vector-valued Calder\'on--Zygmund operators with operator-valued kernel on a space of homogeneous type. This yields the $A_2$-theorem for these Calde\'ron--Zygmund operators, unifying the results from \cite{AV14} and \cite{HH14}.

As an application of this theorem, we will prove a weighted, anisotropic, mixed norm Mihlin multiplier theorem in the next section. We will also use it to study maximal regularity for parabolic partial differential equations in forthcoming work. In these applications $S$ is (a subset of) $\R^d$ equipped with the anisotropic quasi-norm
\begin{equation}\label{eq:anisotropic}
  \abs{s}_{\mbs{a}} := \has{\sum_{j=1}^d\abs{s_j}^{2/a_j}}^{1/2}, \qquad s \in \R^d.
\end{equation}
for some $\mbs{a} \in (0,\infty)^d$ and the Lebesgue measure.

In a different direction our $A_2$-theorem  can be applied in the study of fundamental harmonic analysis operators associated with various discrete and continuous orthogonal expansions, started by Muckenhoupt and Stein \cite{MS65}. In the past decade there has been a surge of results in which such operators are proven to be vector-valued Calder\'on--Zygmund operators on concrete spaces of homogeneous type. Weighted bounds are then often concluded using \cite[Theorem III.1.3]{RRT86} or \cite{RT88}. With our $A_2$-theorem these results can be made quantitative in terms of the $A_p$-characteristic. We refer to \cite{BCN12,BMT07,CGRTV17,NS12,NS07} and the references therein for an overview of the recent developments in this field.

\bigskip

Let $(S,d,\mu)$ be a space of homogeneous type, $X$ and $Y$ be Banach spaces and let $$K\colon(S \times S)\setminus\cbrace{(s,s):s \in S} \to \mc{L}(X,Y)$$
be strongly measurable in the strong operator topology.
We say that $K$ is a \emph{Dini kernel} if there is a $c_K \geq 2$ such that
  \begin{align*}
   \nrm{K(s,t)-K(s,t')} &\leq
   \omega \has{\frac{d(t,t')}{d(s,t)}}\frac{1}{\mu\hab{{B(s,d(s,t))}}},  &&0<d(t,t')\leq \frac{1}{c_K}d(s,t),\\
    \nrm{K(s,t)-K(s',t)} &\leq  \omega \has{\frac{d(s,s')}{d(s,t)}} \frac{1}{\mu\hab{{B(s,d(s,t))}}}, &&0<d(s,s')\leq \frac{1}{c_K}d(s,t),
  \end{align*}
    where $\omega:[0,1]\to [0,\infty)$ is increasing, subadditive, $\omega(0)=0$ and
  \begin{equation*}
    \nrm{K}_{\Dini}:=\int_0^1\omega(t)\frac{\ddn t}{t} <\infty.
  \end{equation*}
  Take $p_0 \in [1,\infty)$ and let
\begin{equation*}
   T \colon L^{p_0}(S;X) \to L^{p_0,\infty}(S;Y)
\end{equation*}
be a bounded linear operator. We say that $T$ has Dini kernel $K$ if for every boundedly supported $f \in L^{p_0}(S;X)$ and a.e. $s \in S \setminus \overline{\supp f}$ we have
\begin{equation*}
  Tf(s) = \int_S K(s,t)f(t)\dd t.
\end{equation*}

\begin{theorem}\label{theorem:A2}
 Let $(S,d,\mu)$ be a space of homogeneous type and let $X$ and $Y$ be Banach spaces. Let $p_0 \in [1,\infty)$ and suppose $T$ is a bounded linear operator from $L^{p_0}(S;X)$ to  $L^{p_0,\infty}(S;Y)$
with Dini kernel $K$.  Then for every boundedly supported $f \in L^1(S;X)$ there exists an $\eta$-sparse collection of cubes $\mc{S}$ such that
\begin{equation*}
    \nrm{Tf(s)}_Y\lesssim_{S,p_0} \,C_T\, \sum_{Q \in \mc{S}} \ipb{\nrm{f}_X}_{1,Q} \ind_{Q}(s), \qquad s \in S.
\end{equation*}
 Moreover, for all $p \in (1,\infty)$ and $w \in A_p$ we have
 \begin{align*}
   \nrm{T}_{L^p(S,w;X) \to L^p(S,w;Y)} &\lesssim_{S,p,p_0} C_T \,[w]_{A_p}^{\max\cbrace{\frac{1}{p-1},1}}
 \end{align*}
 with $C_{T}:= \nrm{T}_{L^{p_0}(S;X)\to L^{p_0,\infty}(S;Y)} + \nrm{K}_{\Dini}$.
\end{theorem}

\begin{proof}
  We will check the assumptions of Theorem \ref{theorem:main} with $p_1=p_2=r=1$. The  weak $L^1$-boundedness of $T$ with
  \begin{equation*}
    \nrm{T}_{L^{1}(S;X)\to L^{1,\infty}(S;Y)} \lesssim_{S,p} C_T.
  \end{equation*}
  follows from the classical Calder\'on-Zygmund argument, see e.g. \cite[Theorem III.1.2]{RRT86}.
  The $1$-sublinearity assumption on $T$ follows from the triangle inequality, so the only thing left to check is the weak $L^1$-boundedness of $\mc{M}_{T,\alpha}^{\#}$.
 Let
 $$\alpha := 3 \,c_d^2 \,\max\cbraceb{\delta^{-1},c_K}$$
  with $c_d$ the quasi-metric constant, $\delta$ as in Proposition \ref{proposition:dyadicsystem} and $c_K$ the constant from the definition of a Dini kernel.
  Fix $s \in S$ and a ball $B=B(z,\rho)$ such that $s \in B$. Then for any $s',s'' \in B$ and $t \in S \setminus \alpha B$ we have
  \begin{align*}
    d(s',t) &\geq \frac{1}{c_d}d(z,t)-d(z,s') \geq  \frac{\alpha \rho}{c_d} -\rho \geq 2 \,c_K \,c_d\, \rho=:\varepsilon\\
    d(s',s'') &\leq 2\,c_d\, \rho = c_K^{-1}\varepsilon
  \end{align*}
Therefore we have for any boundedly supported $f \in L^1(S;X)$
  \begin{align*}
    \nrm{T(\ind_{S \setminus \alpha B}&f)(s') - T_K(\ind_{S \setminus \alpha B}f)(s'')}_{Y} \\
    &\leq  \int_{S\setminus \alpha B}\nrmb{\hab{K(s',t) - K(s'',t)}f(t)}_Y\dd \mu(t)\\
       &\leq   \int_{d(s',t)>\varepsilon}\omega\has{\frac{d(s',s'')}{d(s',t)}}\frac{1}{\mu\hab{{B(s',d(s',t))}}} \nrm{f(t)}_X\dd \mu(t)\\
       &\leq \sum_{j=0}^\infty \omega\hab{c_K^{-1} 2^{-j}} \int_{2^j\varepsilon<d(s',t)\leq 2^{j+1}\varepsilon}\frac{1}{\mu\hab{{B(s',d(s',t))}}}\nrm{f(t)}_X\dd \mu(t)\\
       &\lesssim_S  \sum_{j=0}^\infty\omega\hab{2^{-j-1}} \avint_{B(s',2^{j+1}\varepsilon)}\nrm{f(t)}_X\dd \mu(t)\\
       &\leq \nrm{K}_{\Dini}\,   M\hab{\nrm{f}_X}(s),
  \end{align*}
  where the last step follows from $s \in B(s',2^{j+1}\varepsilon)$ for all $j \in \N$ and
  \begin{equation*}
    \sum_{j=0}^\infty\omega\hab{2^{-j-1}} \leq \sum_{j=0}^\infty\omega\hab{2^{-j-1}} \int_{2^{-j-1}}^{2^{-j}}\frac{\ddn t}{t} \leq \sum_{j=0}^\infty \int_{2^{-j-1}}^{2^{-j}} \omega(t) \frac{\ddn t}{t} = \nrm{K}_{\Dini}.
  \end{equation*}
So taking the supremum over all $s',s'' \in B$ and all balls $B$ containing $s$ we find that $\mc{M}_{T,\alpha}^{\#}f(s) \lesssim_S \nrm{K}_{\Dini}  \, M\hab{\nrm{f}_X}(s)$. Thus by the weak $L^1$-boundedness of the Hardy--Littlewood maximal operator and the density of boundedly supported functions in $L^1(S;X)$ we get
\begin{equation*}
  \nrmb{\mc{M}_{T,\alpha}^{\#}}_{L^{1}(S;X)\to L^{1,\infty}(S;Y)} \lesssim_{S} \nrm{K}_{\Dini}.
\end{equation*}
The pointwise sparse domination now follows from Theorem \ref{theorem:main} and the  weighted bounds from  Proposition \ref{proposition:weights}.
\end{proof}

\begin{remark} In the proof of Theorem \ref{theorem:A2} it actually suffices to use the so-called $L^r$-H\"ormander condition for some $r>1$, which is implied by the Dini condition. See \cite[Section 3]{Li18} for the definition of the $L^r$-H\"ormander condition and a comparison between the $L^r$-H\"ormander and the Dini condition.
\end{remark}

Note that Theorem \ref{theorem:A2} does not assume anything about the Banach spaces $X$ and $Y$ and is therefore applicable in situations where for example $Y=\ell^\infty$. However, in various applications $X$ and $Y$ will need to have the $\UMD$ property in order to check the assumed
      weak $L^{p_0}$-boundedness of $T$ for some $p_0 \in [1,\infty)$. For instance, for a large class of operators the weak $L^{p_0}$-boundedness of $T$ can be checked using theorems like the $T(1)$-theorem  or $T(b)$-theorem. See \cite{Fi90} and \cite{Hy14} for these theorems in the vector-valued setting, which assume the $\UMD$ property for the underlying Banach space.

      If $S$ is Euclidean space, one can also use an (operator-valued) Fourier multiplier theorem to check the a priori $L^{p_0}$-bound, which we will discuss in the next section.

\section{The weighted anisotropic mixed-norm Mihlin multiplier theorem}
     Let $X$ and $Y$ be Banach spaces. Denote the space of $X$-valued Schwartz functions by $\mc{S}(\R^d;X)$ and  the space of $Y$-valued tempered distributions by $\mc{S}'(\R^d;Y):= \mc{L}(\mc{S}(\R^d);Y)$. To an $m \in L^\infty(\R^d;\mc{L}(X,Y))$ we associate the Fourier multiplier operator
\begin{equation*}
  T_m\colon\mc{S}(\R^d;X) \to \mc{S}'(\R^d;Y), \qquad T_m f = (m\widehat{f}\vspace{2pt})^\vee.
\end{equation*}
Since $\mc{S}(\R^d;X)$ is dense in $L^p(\R^d;X)$ and $L^p(\R^d;Y)$ is continuously embedded into $\mc{S}'(\R^d;X)$, one may ask under which conditions on $m$ the operator $T_m$ extends to a bounded operator from $L^p(\R^d;X)$ to $L^p(\R^d;Y)$. If this is the case we call $m$ a bounded Fourier multiplier. We refer to \cite[Chapter 5]{HNVW16} for an introduction to operator-valued Fourier multiplier theory.

One of the main Fourier multiplier theorems is the Mihlin multiplier theorem, first proven in the operator-valued setting by Weis in \cite{We01b}.
The operator-valued Mihlin multiplier theorem of Weis has since been extended in many directions. Recently Fackler, Hyt\"onen and Lindemulder extended the operator-valued Mihlin multiplier theorem to a weighted, anisotropic, mixed norm setting in \cite{FHL18}. This is for example useful in the study of spaces of smooth, vector-valued functions and has applications to parabolic PDEs with inhomogeneous boundary conditions, see e.g. \cite{Li17b}.
 In \cite{FHL18} the Mihlin multiplier theorem is shown using the following two approaches:
\begin{itemize}
  \item Using a weighted Littlewood--Paley decomposition, they show a weighted, anisotropic, mixed-norm
  Mihlin multiplier theorem for rectangular $A_p$-weights, i.e. $A_p$-weights for which the defining supremum is taken over rectangles instead of balls.
  \item  Using Calder\'on--Zygmund theory, they show a weighted, isotropic, non-mixed-norm
  Mihlin multiplier theorem for cubicular $A_p$-weights, i.e. $A_p$-weights for which the defining supremum is taken over cubes, which is equivalent to the definition using balls we used
in Section \ref{section:weights}.
\end{itemize}
Both approaches have their pros and cons. The result using a Littlewood--Paley decomposition only requires estimates of $\partial^\theta m$ for $\theta \in \cbrace{0,1}$, whereas the approach using Calder\'on--Zygmund theory also requires estimates of higher-order derivatives. On the other hand,
the class of rectangular $A_p$-weights is a proper subclass of the class of cubicular $A_p$-weights.

In applications it is be desirable to have the Mihlin multiplier theorem for cubicular $A_p$-weights in the anisotropic, mixed-norm setting as well. This would remove the need to distinguish between the isotropic and anisotropic setting in e.g. \cite[(6) on p.64]{Li17b}.
In order to obtain the Mihlin multiplier theorem for cubicular $A_p$-weights in the anisotropic, mixed-norm setting one needs Calder\'on--Zygmund theory in $\R^d$ equipped with an anisotropic norm. Since this is a special case of a space of homogeneous type, we can use Theorem \ref{theorem:A2} to supplement the results of \cite{FHL18}, which will be the main result of this section.

\bigskip

Let us introduce the anisotropic, mixed-norm setting. For $\mbs{a} \in (0,\infty)^d$ let $\abs{\,\cdot\,}_{\mbs{a}}$ be the anisotropic quasi-norm as in \eqref{eq:anisotropic} and define $$\R^d_{\mbs{a}} := (\R^d, \abs{\,\cdot-\cdot\,}_{\mbs{a}}, \dd t),$$ where $\dd t$ denotes Lebesgue measure. Then $\R^d_{\mbs{a}}$ is a space of homogeneous type and e.g.
  \begin{align*}
  \ms{D}&:= \cbraces{\prod_{j=1}^d \hab{2^{-a_jn}([0,1) +m_j)} :\mbs{m}\in \Z^d, n \in \Z}
\end{align*}
is a dyadic system in $\R^d_{\mbs{a}}$.
 We write
$\abs{\mbs{a}}_1:=\sum_{j=1}^d a_j,$ $\abs{\mbs{a}}_\infty := \max_{j=1,\ldots,d}a_j,$
and for $\theta \in \N^d$ we set $\mbs{a} \cdot\theta := \sum_{j=1}^d a_j \theta_{j}$.

Take $l \in \N$, $\mathpzc{d} \in \N^l$ and consider the $\mathpzc{d}$-decomposition of $\R^d$:
\begin{equation*}
\R^{d}_{\mz{d}}:= \R^{\mathpzc{d}_1} \times \ldots \times \R^{\mathpzc{d}_l}.
\end{equation*}
For a $t \in \R^{d}_{\mz{d}}$ we write $t = (t_1,\ldots,t_l)$ with $t_j \in \R^{\mz{d}_j}$ for $j=1,\ldots,l$ and similarly we write $\mbs{a} = (\mbs{a}_1,\ldots,\mbs{a}_l)$.
For $\mbs{p} \in [1,\infty)^l$, a vector of weights $\mbs{w}  \in \prod_{j=1}^l A_p(\R^{\mathpzc{d}_j}_{
{\mbs{a}_j}})$ and a Banach space $X$ we define the weighted mixed-norm Bochner space $L^{\mbs{p}}(\R^{d}_{\mz{d}},\mbs{w};X)$ as the space of all strongly measurable $f:\R^{d}_{\mz{d}} \to X$ such that
\begin{equation*}
  \nrm{f}_{L^{\mbs{p}}(\R^{d}_{\mz{d}},\mbs{w};X)}:= \has{\int_{\R^{\mz{d}_1}}\ldots\has{\int_{\R^{\mz{d}_l}}\nrm{f}_X^{p_l} w_l\dd t_l}^{\frac{p_{l-1}}{p_l}}\ldots w_1 \dd t_1}^{\frac{1}{p_1}}
\end{equation*}
is finite.

We are now ready to state and prove the announced weighted anisotropic, mixed-norm Mihlin multiplier theorem.

\begin{theorem}\label{theorem:mihlin}
Let $X$ and $Y$ be $\UMD$ Banach spaces, set $N = \abs{\mbs{a}}_1+ \abs{\mbs{a}}_\infty+ 1$
 and let $m \in L^\infty(\R^d;\mc{L}(X,Y))$. Suppose that for all $\theta \in \N^d$ with $\mbs{a} \cdot\theta \leq N$ the distributional derivative $\partial^\theta m$ coincides with a continuous function on $\R^d \setminus \cbrace{0}$ and we have the $\mc{R}$-bound
\begin{equation*}
 \mc{R}\hab{\cbraceb{\abs{\xi}_{\mbs{a}}^{\mbs{a} \cdot \theta}\cdot  \partial^\theta m(\xi): \xi \in \R^d}\setminus\cbrace{0}} \leq C_m.
\end{equation*}
for some $C_m>0$. Then for every compactly supported $f \in L^1(\R^d;X)$ there exists an $\eta$-sparse collection of anisotropic cubes $\mc{S}$ such that
\begin{equation*}
    \nrm{T_mf(s)}_Y\lesssim_{X,Y,\mbs{a}} \,C_m\, \sum_{Q \in \mc{S}} \ipb{\nrm{f}_X}_{1,Q} \ind_{Q}(s), \qquad s \in \R^d.
\end{equation*}
Moreover, for all $\mbs{p} \in (1,\infty)^l$ and $\mbs{w} \in  \prod_{j=1}^l A_{p_j}(\R^{\mz{d}_j}_{\mbs{a}_j})$ we have
\begin{align*}
  \nrm{T_m}_{L^{\mbs{p}}(\R^d_{\mz{d}},\mbs{w};X) \to L^{\mbs{p}}(\R^d_{\mz{d}},\mbs{w};Y)} &\lesssim_{X,Y,\mz{d},\mbs{a},\mbs{p}, \mbs{w}} C_m.
\end{align*}
\end{theorem}

\begin{proof}
 We will check the conditions of Theorem \ref{theorem:A2}. By \cite[Theorem 3]{Hy07}, which trivially extends to the case $X \neq Y$,
  we know that $T_m$ is bounded from $L^2(\R^d;X)$ to $L^2(\R^d;Y)$ with
  \begin{equation*}
    \nrm{T_m}_{L^2(\R^d;X) \to L^2(\R^d;Y)} \lesssim_{X,Y,\mz{d},\mbs{a}} C_m.
  \end{equation*}
  By \cite[Lemma 4.4.6 and 4.4.7]{Li14b} we know that $\widecheck{m}$ coincides with a continuous function on $\R^d \setminus \cbrace{0}$, which is bounded away from $0$ and $$K(t,s):= \widecheck{m}(t-s), \qquad t\neq s$$  is a Dini kernel on the space of homogeneous type  $\R^d_{\mbs{a}}$ with
  \begin{equation*}
    \omega(r) = C_{\mbs{a}} \cdot C_m \cdot r^{\min \mbs{a}}, \qquad r \in [0,1].
  \end{equation*}
  Now let $f \in L^p(\R^d;X)$ with compact support. Fix a $c \in \R^d \setminus \overline{\supp f}$ and take $r>0$ such that $B(c,2r) \cap \overline{\supp f} = \varnothing$. Take a sequence $(f_n)_{n=1}^\infty$ in $\mc{S}(\R^d;X)$ such that $\overline{\supp f_n}  \cap B(c,r)= \varnothing$ and $f_n \to f$ in $L^2(\R^d;X)$. Then $Tf_n \to Tf$ in $L^2(\R^d;X)$ and, by passing to a subsequence if necessary, we have $f_n(t) \to f(t)$ and $Tf_n(t) \to Tf(t)$ for a.e. $t \in \R^d$. Fix $n \in \N$, then we have for all $\varphi \in C_c^\infty(\R^d \setminus \overline{\supp f_n})$
  \begin{align*}
    \ip{T_m f_n, \varphi}  &= \int_{\R^d} m(s)\widehat{f_n}(s) \widecheck{\varphi}(s)\dd s\\
    &= \int_{\R^d}\widecheck{m}(s)   \int_{\R^d}f_n(t-s) \varphi(t) \dd t\dd s \\
    &= \int_{\R^d} \int_{\R^d} K(t,s)f_n(s)\dd s \, \varphi(t) \dd t
  \end{align*}
  from which we obtain for a.e. $t \in B(c,r)$
  \begin{equation*}
    T_m f(t) = \lim_{n\to \infty} T_m f_n(t) =  \lim_{n \to \infty} \int_{\R^d}K(t,s)f_n(s)\dd s = \int_{\R^d}K(t,s)f(s)\dd s
  \end{equation*}
  Covering $\R^d \setminus \overline{\supp f}$ by countably many such balls, we conclude that $T_m$ has kernel $K$. Therefore the sparse domination, as well as the weighted estimate in case $l=1$, follows from Theorem \ref{theorem:A2}.

To conclude the proof we will show the case $l=2$, the general case follows by iterating the argument. Take $\mbs{p} \in (1,\infty)^2$ and $\mbs{w} \in A_{p_1}(\R^{\mz{d}_1}_{\mbs{a}_1}) \times A_{p_2}(\R^{\mz{d}_2}_{\mbs{a}_2})$. For $v_1 \in A_{p_2}(\R^{\mz{d}_1}_{\mbs{a}_1})$ note that
 $${v}(t):= v_1(t_1)\cdot  w_2(t_2) , \qquad t \in \R^{\mz{d}_1} \times \R^{\mz{d}_2}$$ belongs to $A_{p_2}(\R^d_{\mbs{a}})$, so by the case $l=1$ we have
   \begin{equation*}
      \nrm{T_mf}_{L^{p_2}(\R^d,{v};Y)} \lesssim_{X,Y,\mz{d},\mbs{a},p_2, {v}} C_m \cdot \nrm{f}_{L^{p_2}(\R^d,{v};X)}
   \end{equation*}
   for all $f \in L^{p_2}(\R^d, {v};X)$.
Since  balls in $\R^{\mz{d}_2}$ with respect to the quasi-metric $\abs{\,\cdot-\cdot\,}_{\mbs{a}_2}$ form a Muckenhoupt basis, we can use Rubio de Francia extrapolation as in \cite[Theorem 3.9]{CMP11} on the extrapolation family
  \begin{equation*}
    \cbraces{\hab{\nrm{T_mf}_{L^{p_2}(\R^{\mz{d}_2},w_2;Y)},\nrm{f}_{L^{p_2}(\R^{\mz{d}_2},w_2;X)}}:f \colon \R^d \to X \text{ simple}}
  \end{equation*}
to deduce
\begin{equation*}
  \nrm{T_mf }_{L^{\mbs{p}}(\R^{d}_{\mz{d}} ,\mbs{w};Y)} \lesssim_{X,Y,\mz{d},\mbs{a},\mbs{p},\mbs{w}} C_m \nrm{f }_{L^{\mbs{p}}(\R^{d}_{\mz{d}} ,\mbs{w};X) }
\end{equation*}
for all simple $f$, which implies the result by density.
\end{proof}

\begin{remark}~
\begin{enumerate}[(i)]
\item The weight dependence of the implicit constant in Theorem \ref{theorem:mihlin} in the case $l=1$ is $[w]_{A_{p}(\R^d_{\mbs{a}})}^{\max\cbrace{\frac{1}{p-1},1}}$, which is sharp. For $l\geq 2$ the dependence our proof yields is more complicated and not sharp for all choices of $\mbs{p} \in (1,\infty)^l$.
  \item In the proof of Theorem \ref{theorem:mihlin} we only use the $\mc{R}$-boundedness of the set
  $$
  \cbraceb{\abs{\xi}_{\mbs{a}}^{\mbs{a} \cdot \theta}\cdot  \partial^\theta m(\xi): \xi \in \R^d\setminus\cbrace{0}}
   $$
   for $\theta  \in \cbrace{0,1}^d$. For all other $\theta \in \N^d$ with $\mbs{a} \cdot \theta \leq N$ it suffices to know uniform boundedness of this set.
\item One could reduce the number of derivatives necessary in Theorem \ref{theorem:mihlin}, by arguing as in \cite{Hy04} instead of using \cite[Lemma 4.4.6 and 4.4.7]{Li14b}. See also \cite[Section 6]{FHL18}.
\item Using the sparse domination of Theorem \ref{theorem:mihlin} one can also deduce two-weight estimates for $T_m$ as in \cite[Section 6]{FHL18}.
\end{enumerate}
\end{remark}

\section{The Rademacher maximal function}\label{section:maximal}
In this section we will apply Theorem \ref{theorem:localsparse} to the Rademacher maximal function. The proofs will illustrate very nicely how the geometry of the Banach space plays a role in deducing the localized $\ell^r$-estimate for this operator. In particular, we will use the type of a Banach space $X$ to deduce the localized $\ell^r$-estimate for the Rademacher maximal function.

 The Rademacher maximal function was introduced by Hyt\"onen, McIntosh and Portal in \cite{HMP08} as a vector-valued generalization of Doob's maximal function that takes into account the different ``directions'' in a Banach space. They used the Rademacher maximal function to prove a Carleson's embedding theorem for vector-valued functions in connection to Kato's square root problem in Banach spaces. The Carleson's embedding theorem for vector-valued functions has since found many other applications, like the local vector-valued $T(b)$ theorem (see \cite{HV15}).

Let $(S,d,\mu)$ be a space of homogeneous type with a dyadic system $\ms{D}$ and let $X$ be a Banach space.
For $f \in L^1_{\loc}(S;X)$ we define the \emph{Rademacher maximal function} by
\begin{align*}
  M_{\Rad}^{\ms{D}}f(s)&:= \sup\cbraces{\nrms{\sum_{Q \in \ms{D}:s \in Q}\varepsilon_Q\lambda_Q \ip{f}_{1,Q}}_{L^2(\Omega;X)}: \\&\hspace{2cm}(\lambda_Q)_{Q \in \ms{D}} \text{ finitely non-zero with } \sum_{Q \in \ms{D}}\abs{\lambda_Q}^2\leq 1},
\end{align*}
where $(\varepsilon_Q)_{Q \in \ms{D}}$ is a Rademacher sequence on $\Omega$. One can interpret this maximal function as Doob's maximal function
 \begin{equation*}
  f^*(s): =\sup_{Q \in \ms{D}:s \in Q} \nrmb{\ip{f}_{1,Q}}_X, \qquad s\in S,
\end{equation*}
 with the uniform bound over the $\ip{f}_{1,Q}$'s replaced by the $\mc{R}$-bound. Here the $\mc{R}$-bound of a set $U\subseteq X$ is the $\mc{R}$-bound of the family of operators $T_x:\C \to X$ given by $\lambda \mapsto \lambda x$ for $x \in U$.

We say that the Banach space $X$ has the $\RMF$ property if $M_{\Rad}^{\ms{D}[0,1)}$ is a bounded operator on $L^p([0,1);X)$ for some $p \in (1,\infty)$, where
$$\ms{D}[0,1) := \cbraceb{2^{-k} [j-1,j): k\in \N\cup\cbrace{0},\, j=1,\ldots,2^k}$$
is the standard dyadic system in $[0,1)$.
It was shown by Hyt\"onen, McIntosh and Portal \cite[Proposition 7.1]{HMP08} that this implies boundedness for all $p \in (1,\infty)$ and by Kemppainen \cite[Theorem 5.1]{Ke11}  that  this implies boundedness of $M_{\Rad}^{\ms{D}}$  on $L^p(S;X)$ for any space of homogeneous type $(S,d,\mu)$ with a dyadic system $\ms{D}$.

The relation of  $\RMF$ property to other Banach space properties is not yet fully understood. However, we do have some necessary and sufficient conditions:
\begin{itemize}
  \item The $\mc{R}$-bound of a set $U\subseteq X$ is equivalent to the uniform bound of that set if and only if $X$ has type $2$ (see \cite[Proposition 8.6.1]{HNVW17}). Therefore if $X$ has type $2$ we have for any $f \in L^1_{\loc}([0,1);X)$ that $M_{\Rad}^{\ms{D}[0,1)}f \lesssim M^{\ms{D}[0,1)}(\nrm{f}_X)$, so $X$ has the $\RMF$ property.
  \item Any $\UMD$ Banach lattice has the $\RMF$ property, see also the discussion related to the Hardy--Littlewood maximal operator at the end of this section.
  \item Non-commutative $L^p$-spaces for $p \in (1,\infty)$ have the $\RMF$ property, see \cite[Corollary 7.6]{HMP08}.
  \item The $\RMF$ property implies nontrivial type, see \cite[Proposition 4.2]{Ke11}.
\end{itemize}
It is an open problem whether nontrivial type or even the $\UMD$ property implies the $\RMF$ property.

\bigskip

Weighted bounds for the Rademacher maximal function in the Euclidean setting were studied by Kemppainen \cite[Theorem 1]{Ke13}. The proof was based on a good-$\lambda$ inequality, which does not give sharp quantitative estimates in terms of the weight characteristic. Using Theorem \ref{theorem:localsparse} we can prove sharp quantitative weighted estimates for the Rademacher maximal function through sparse domination. We will not consider the situation in which $X$ has type $2$, as this case follows directly from
$M_{\Rad}^{\ms{D}[0,1)}f \lesssim M^{\ms{D}[0,1)}(\nrm{f}_X)$ and the well-known sparse domination for the Hardy--Littlewood maximal operator.

We will need a version of the Rademacher maximal function for finite collections of cubes. For a subcollection of cubes $\mc{D} \subseteq \ms{D}$ we define  $M_{\Rad}^{\mc{D}}$ analogous to $M_{\Rad}^{\ms{D}}$.

\begin{theorem}\label{theorem:RMF}
  Let $(S,d,\mu)$ be a space of homogeneous type with a dyadic system $\ms{D}$ and let $X$ be a Banach space with the $\RMF$ property. Assume that $X$ has type $r$ for $r \in [1,2)$. For any finite collection of cubes $\mc{D} \subseteq \ms{D}$ and  $f \in L^1(S;X)$ there exists an $\frac12$-sparse collection of cubes $\mc{S}\subseteq \ms{D}$ such that
  \begin{equation*}
    M_{\Rad}^{\mc{D}}f(s) \lesssim_{X,S,\ms{D},r} \has{\sum_{Q \in \mc{S}}\ipb{\nrm{f}_X}_{1,Q}^{(\frac1r-\frac12)^{-1}}\ind_Q(s)}^{\frac1r-\frac12}, \qquad s\in S
  \end{equation*}
Moreover, for all $p \in (1,\infty)$ and $w\in A_p$ we have
  \begin{align*}
    \nrmb{M_{\Rad}^{\ms{D}}}_{L^p(S,w;X) \to L^p(S,w;X)} &\lesssim_{X,S,\ms{D},p,r} [w]_{A_p}^{\max\cbraceb{\frac{1}{p-1},\frac{1}{r}-\frac{1}{2}}}.
  \end{align*}
\end{theorem}

\begin{proof}
Fix a finite collection of cubes $\mc{D} \subseteq \ms{D}$. By \cite[Proposition 6.1]{Ke11} $M_{\Rad}^{\mc{D}}$  is weak $L^1$-bounded.
We will view $M_{\Rad}^{\mc{D}}$ as a bounded operator
  \begin{equation*}
    M_{\Rad}^{\mc{D}}: L^1(S;X) \to L^{1,\infty}(S; \mc{L}(\ell^2(\mc{D}),L^2(\Omega;X)))
  \end{equation*}
  given by
  \begin{equation*}
    M_{\Rad}^{\mc{D}}f(s) = \has{(\lambda_Q)_{Q \in \mc{D}} \mapsto \sum_{Q \in \ms{D}:s \in Q} \varepsilon_Q \lambda_Q \,\ip{f}_{1,Q}}, \qquad s \in S,
  \end{equation*}
  where $(\varepsilon_Q)_{Q \in \mc{D}}$ is a Rademacher sequence on $\Omega$.

  For $Q \in \ms{D}$ set
  $$\mc{D}(Q):=\cbrace{P \in\mc{D}:P \subseteq Q}$$
   and define $T_Q:= M_{\Rad}^{\mc{D}(Q)}$. Then $\cbrace{T_Q}_{Q \in \ms{D}}$ is a $1$-localization family for $M_{\Rad}^{\mc{D}}$.
 Furthermore we have for $f \in L^1(S;X)$ and $s \in Q \in \ms{D}$ that
  \begin{align*}
    \mc{M}^{\#}_{M_{\Rad}^{\mc{D}},Q}f(s)  &= \sup_{\substack{Q' \in \ms{D}(Q):\\s \in Q'}}  \esssup_{s' ,s'' \in Q'} \nrmb{T_{Q\setminus Q'}f(s')-T_{Q\setminus Q'}f(s'')}_{\mc{L}(\ell^2(\ms{D}),L^2(\Omega;X))}\\&=0
  \end{align*}
  where the second step follows from the fact that $T_{Q\setminus Q'}f = M_{\Rad}^{\mc{D}(Q)\setminus \mc{D}(Q')}f$ is constant on $Q'$. So $\mc{M}^{\#}_{M_{\Rad}^{\mc{D}},Q}$ is trivially bounded from $L^1(S;X)$ to $L^{1,\infty}(S)$.

Set  $q:=(\frac1r-\frac12)^{-1}$. To check the localized $\ell^q$-estimate for $M_{\Rad}^{\mc{D}}$ take $Q_1,\ldots,Q_n \in \ms{D}$ with $Q_n\subseteq \ldots\subseteq Q_1$. Let $(\lambda_Q)_{Q \in \mc{D}} \in \ell^2(\mc{D})$ be of norm one and let $(\varepsilon_Q)_{Q \in \mc{D}}$ and $(\varepsilon'_k)_{k=1}^n$ be Rademacher sequences on $\Omega$ and $\Omega'$ respectively. Define for $k=1,\ldots,n-1$
\begin{equation*}
  \lambda_k := \has{\sum_{Q \in \mc{D}\ha{Q_{k+1}}\setminus \mc{D}\ha{Q_k}}
\abs{\lambda_Q}^2}^{1/2}, \qquad \lambda_{n} := \has{\sum_{Q \in \mc{D}\ha{Q_n}} \abs{\lambda_Q}^2}^{1/2}
\end{equation*}
 Then for $f \in L^1(S;X)$, setting $f_Q := \varepsilon_Q \lambda_Q\ip{f}_{1,Q}$,  we have
  \begin{align*}
    \nrms{\sum_{Q \in \mc{D}\ha{Q_1}} &\varepsilon_Q \lambda_Q\ip{f}_{1,Q}}_{L^2(\Omega;X)}\\
    &= \nrms{\varepsilon_n'\sum_{Q \in \mc{D}\ha{Q_n}} f_Q+ \sum_{k=1}^{n-1}\varepsilon_{k}'\sum_{Q\in \mc{D}\ha{Q_{k+1}}\setminus \mc{D}\ha{Q_k}} f_Q}_{L^2(\Omega\times \Omega';X)}\\
    &\lesssim_{X,r} \has{\lambda_n^r \nrms{\sum_{Q \in \mc{D}\ha{Q_n}}  \lambda_n^{-1}f_Q}^r_{L^2(\Omega;X)} \\&\hspace{1cm}+ \sum_{k=1}^{n-1}\lambda_{k}^r \nrms{\sum_{Q\in \mc{D}\ha{Q_{k+1}}\setminus \mc{D}\ha{Q_k}} \lambda_{k}^{-1}f_Q}_{L^2(\Omega;X)}^r}^{1/r}\\
    &\leq\has{ \nrms{\sum_{Q \in \mc{D}\ha{Q_n}}  \varepsilon_Q \lambda_n^{-1} \lambda_Q\ip{f}_{1,Q}}^{q}_{L^2(\Omega;X)} \\&\hspace{1cm}+ \sum_{k=1}^{n-1} \nrms{\sum_{Q\in \mc{D}\ha{Q_{k+1}}\setminus \mc{D}\ha{Q_k}} \varepsilon_Q \lambda_{k}^{-1}\lambda_Q\ip{f}_{1,Q}}_{L^2(\Omega;X)}^{q}}^{1/q},
  \end{align*}
using randomization (see \cite[Proposition 6.1.11]{HNVW17}) in the first step, type $r$ of $X$ in the second step, and H\"older's inequality and $\sum_{k=1}^n \lambda_k^2=1$ in the last step. Noting that for $k=1,\ldots,n-1$
\[\sum_{Q\in \mc{D}(Q_{k+1})\setminus \mc{D}(Q_k)} \abs{\lambda_{k}^{-1}{\lambda_Q}}^2=1,\qquad \sum_{Q\in {\mc{D}(Q_n)} } \abs{\lambda_{n}^{-1}{\lambda_Q}}^2=1,\] this implies the  localized $\ell^{q}$-estimate for $M_{\Rad}^{\mc{D}}$.

Having checked all assumptions of Theorem \ref{theorem:localsparse} for $M_{\Lat}^{\mc{D}}$ it follows that for any $Q \in \mc{D}$ there is a $\frac{1}{2}$-sparse collection of cubes $\mc{S}_Q\subseteq \ms{D}(Q)$ such that
\begin{align*}
  \nrmb{T_Q(s)}_Y\lesssim_{X,S,\ms{D},r} ,
  \has{ \sum_{P \in \mc{S}} \ipb{\nrm{f}_X}_{p,\alpha P}^r \ind_P(s)}^{1/r},\qquad s \in Q.
  \end{align*}
Let $\mc{D}'$ be the maximal cubes (with respect to set inclusion) of $\mc{D}$, which are pairwise disjoint. Then $\mc{S}:=\bigcup_{Q \in \mc{D}'} \mc{S}_Q$ is a $\frac{1}{2}$-sparse collection of cubes that satisfies the claimed sparse domination
as $T_Q(s) = M_{\Rad}^{\mc{D}}f(s)$ for any $s\in Q \in \mc{D}'$ and $M_{\Rad}^{\mc{D}}f$ is zero outside $\bigcup_{Q \in \mc{D}'}Q$.  The weighted bounds follow from Proposition \ref{proposition:weights} and the monotone convergence theorem.
\end{proof}

Let us check that the weighted estimate in Theorem \ref{theorem:RMF}, and consequently also the sparse domination in Theorem \ref{theorem:RMF}, is sharp. We take $X=\ell^r$ for $r \in (1,2)$, a prototypical Banach space with type $r$.  Since $\mc{R}$-bounds are stronger than uniform bounds, we note that for any strongly measurable $f \colon [0,1) \to \ell^q$ we have
\begin{equation*}
 f^*(s) \leq M_{\Rad}^{\ms{D}[0,1)}f(s),\qquad s \in [0,1).
\end{equation*}
Thus by the corresponding result for Doob's maximal operator (see \cite[Proposition 3.2.4]{HNVW16}), we have for $p \in (1,\infty)$
\begin{equation}\label{eq:radlowerbound}
  \nrmb{M_{\Rad}^{\ms{D}[0,1)}}_{L^p([0,1);\ell^r) \to L^p([0,1);\ell^r) } \geq \frac{p}{p-1}
\end{equation}
Now let $(e_n)_{n=1}^\infty$ be the canonical basis of $\ell^r$ and define
\begin{equation*}
  f(s):= \sum_{n=1}^\infty \ind_{[2^{-n},2^{-n+1})}(s) e_n, \qquad s \in [0,1).
\end{equation*}
For $p \in (1,\infty)$ we have
\begin{equation*}
  \nrm{f}_{L^p([0,1);\ell^r)} = 1.
\end{equation*}
To compute $\nrm{M_{\Rad}^{\ms{D}[0,1)}f}_{L^p([0,1);\ell^r)}$ set $I_j:= [0,2^{-j+1}]$, take $s \in $ and let $m \in \N$ be such that $2^{-m}\leq s \leq 2^{-m+1}$. Then we have, using $\lambda_{I_j} = m^{-1/2}$ for $j=1,\ldots,m$ and the Khintchine--Maurey inequalities (see \cite[Theorem 7.2.13]{HNVW17}), that
\begin{align*}
  M_{\Rad}^{\ms{D}[0,1)}f(s) &\geq \frac{1}{m^{1/2}}  \nrms{\sum_{j=1}^m \varepsilon_j\ip{f}_{1,I_j}}_{L^2(\Omega;\ell^r)}
  \gtrsim \frac{1}{m^{1/2}}\nrms{\has{\sum_{j=1}^m \ip{f}_{1,I_j}^2}^{1/2}}_{\ell^r}\\
  &\gtrsim \frac{1}{m^{1/2}}\nrms{\sum_{j=1}^m e_j}_{\ell^r}
  \gtrsim m^{1/r-1/2} \gtrsim {\log(1/s)}^{1/r-1/2}.
\end{align*}
Therefore we obtain
\begin{align*}
  \nrmb{M_{\Rad}^{\ms{D}[0,1)}f}_{L^p([0,1);\ell^r)}&\gtrsim \has{\int_0^1 {\log(1/s)}^{p/r-p/2} \dd s}^{1/p}
  \\&= \has{\int_1^\infty x^{p/r-p/2} \ee^{-x}\dd x}^{1/p} \\&\geq \has{\sum_{n=2}^\infty n^{p/r-p/2}\ee^{-n}}^{1/p} \\&\gtrsim p^{1/r-1/2},
\end{align*}
where we drop all terms except $n = \ceil{p}$ in the last step. Thus combined with \eqref{eq:radlowerbound} we find
\begin{equation*}
  \nrmb{M_{\Rad}^{\ms{D}[0,1)}}_{L^p([0,1);\ell^r) \to L^p([0,1);\ell^r) } \gtrsim \max\cbraces{\frac{1}{p-1}, p^{1/r-1/2}},
\end{equation*}
which implies that the weighted estimate in Theorem \ref{theorem:RMF} is sharp by \cite[Theorem 1.2]{LPR15}.

\bigskip

To finish this section we will compare the sparse domination for the Rademacher maximal operator in Theorem \ref{theorem:RMF} with the sparse domination for the lattice Hardy--Littlewood maximal operator obtained by H\"anninnen and the author in \cite[Theorem 1.3]{HL17}. Let $X$ be a Banach lattice with finite cotype and $\ms{D}$ the standard dyadic system in $\R^d$. For a simple function $f:\R^d \to X$ define  \emph{dyadic lattice Hardy--Littlewood maximal operator} (see e.g. \cite{GMT93}) by
\begin{equation}\label{eq:maxop}
  M_{\Lat}^{\ms{D}}f(s) : = \sup_{Q \in \ms{D}: s \in Q} \ipb{\abs{f}}_{1,Q}, \qquad s\in \R^d,
\end{equation}
where the absolute value and the supremum are taken in the lattice sense. By the Khintchine--Maurey inequalities (see e.g. \cite[Theorem 7.2.13]{HNVW17}) we have
      $$M_{\Rad}^{\ms{D}}f \lesssim M_{\Lat}^{\ms{D}}f$$
       for any simple $f\colon \R^d \to X$.
By \cite{Bo84,Ru86} we know  that $X$ has the $\UMD$ property if and only if
 $M_{\Lat}^{\ms{D}}$ is bounded on $L^p(\R^d;X)$ and $L^p(\R^d;X^*)$ for some (all) $p \in (1,\infty)$, which  implies that any $\UMD$ Banach lattice has the $\RMF$ property.

 Comparing the sparse domination result in  Theorem \ref{theorem:RMF} with the corresponding sparse domination result for the dyadic lattice Hardy--Littlewood maximal operator, we see that the sparse operator in Theorem \ref{theorem:RMF} is smaller than the sparse operator in \cite[Theorem 1.3]{HL17}. Moreover, the sparse domination for the lattice Hardy--Littlewood maximal operator  is sharp, as shown in \cite[Theorem 1.2]{HL17}. Therefore on any $\RMF$ Banach lattice that is not $\infty$-convex, the operators $M_{\Rad}^{\ms{D}}$ and $M_{\Lat}^{\ms{D}}$ are incomparable,  i.e. the (dyadic) lattice Hardy--Littlewood maximal operator is strictly larger than the Rademacher maximal operator. As the only $\infty$-convex $\RMF$ Banach lattices are the finite dimensional ones, we have the following corollary.

\begin{corollary}
  Let $X$ be an infinite dimensional $\RMF$ Banach lattice. Then there is no $C>0$ such that for all simple $f\colon \R^d \to X$
  $$
    M_{\Lat}^{\ms{D}}f\leq C \, M_{\Rad}^{\ms{D}}f.$$
\end{corollary}

\section{Further Applications}\label{section:further}
 In this final section we comment on some further applications of our main theorems, for which we leave the details to the interested reader.
\begin{itemize}
\item Sparse domination and weighted bounds for variational truncations of Calder\'on--Zygmund operators were studied in \cite{FZ16, HLP13b, MTX15, MTX17}. The arguments presented in these references also imply the boundedness of our sharp grand maximal truncation operator and thus by Theorem \ref{theorem:main} yield sparse domination of the variational truncations of Cal\-der\'on--Zygmund operators.
\item In \cite{LOR17} Lerner, Ombrosi and Rivera-R\'ios show sparse domination for commutators of a $\BMO$ function $b$ with a Calder\'on--Zygmund operator using sparse operators adapted to the function $b$. By a slight adaptation of the arguments presented in the proof of Theorem \ref{theorem:localsparse}, one can prove the main result of \cite{LOR17} in our framework and extend it to the vector-valued setting and to spaces of homogeneous type.
\item H\"ormander--Mihlin type conditions as in \cite[Theorem IV.3.9]{GR85} imply the weak $L^{p_1}$-boundedness  of our maximal truncation operator for $p_1>n/a$ and thus sparse domination for the associated Fourier multiplier operator by Theorem \ref{theorem:main}. Vector-valued extensions under Fourier type assumptions can be found in \cite{GW03,Hy04} and Theorem \ref{theorem:main} may therefore also be used to prove weighted results in that setting.
\item In \cite{Le11} Lerner used his local mean oscillation decomposition to deduce sparse domination and sharp weighted norm inequalities for various Littlewood--Paley operators. These results are also an almost immediate consequence of Theorem \ref{theorem:localsparse} with $r=2$, using a truncation of the cone of aperture in the definition of a Littlewood--Paley operator in order to make the localized $\ell^2$-estimate checkable. Using similar arguments one can also treat the dyadic square function with Theorem \ref{theorem:localsparse}, which yields the sharp weighted norm inequalities as obtained by Cruz-Uribe, Martell and Perez \cite{CMP12}.

 Very recently Bui and Duong \cite{BD19} extended the results in \cite{Le11} to square functions of a general operator $L$ which has a Gaussian heat kernel bound and a bounded holomorphic functional calculus on $L^2(S)$, where $(S,d,\mu)$ is a space of homogeneous type. The arguments they present can also be used to estimate our sharp grand maximal truncation operator, so their result is also be treated by Theorem \ref{theorem:localsparse}.
\item Fackler, Hyt\"onen and Lindemulder \cite{FHL18} proved weighted vector-valued Littlewood-Paley theory on a $\UMD$ Banach space in order to prove their weighted, anisotropic, mixed-norm Mihlin multiplier theorems. Using Theorem \ref{theorem:main} and Proposition \ref{proposition:weights} on the Littlewood--Paley square function with smooth cut-offs one can prove sparse domination and weighted estimates in the smooth cut-off case. This can then be transferred to sharp cut-offs by standard arguments, recovering \cite[Theorem 3.4]{FHL18}.
\item  In \cite{PSX12} Potapov, Sukochev and Xu proved extrapolation upwards of unweighted vector-valued Littlewood--Paley--Rubio de Francia inequalities. Using \cite[Lemma 4.5]{PSX12}  one can check the weak $L^2$-boundedness of our sharp grand maximal truncation operator, which by Theorem \ref{theorem:main} and Proposition \ref{proposition:weights} yields sparse domination and weighted estimates for vector-valued Littlewood--Paley--Rubio de Francia estimates. In the scalar case sparse domination was shown by Garg, Roncal and Shrivastava \cite{GRS18} using time-frequency analysis.
 \item Theorem \ref{theorem:fractional} can be used to show sparse domination and sharp weighted estimates for fractional integral operators as in \cite{CB13, CB13b,Cr17,IRV18}. The boundedness of the sharp grand maximal truncation operator associated to these operators can be shown using a similar argument as we used in the proof of Theorem \ref{theorem:A2}.
\item In \cite{BFP16} Bernicot, Frey and Petermichl show that the sparse domination
principle is also applicable to non-integral singular operators falling
outside the scope of Calder\'on--Zygmund operators. Sparse domination for square functions related to these operators was studied in \cite{BBR20}.  The methods developed in these papers actually show the boundedness of the localized sharp grand $q$-maximal truncation operator used in Theorem \ref{theorem:sparseform}, so these results also fit in our framework.
\end{itemize}

\subsection*{Acknowledgement} The author would like to thank Dorothee Frey, Bas Nieraeth and Mark Veraar for their helpful comments on the draft version of this paper. Moreover the author would like to thank Luz Roncal for bringing one of the applications in Section \ref{section:A2} under the author's attention and Olli Tapiola for his remarks on the Lebesgue differentiation theorem in spaces of homogeneous type.

\newcommand{\etalchar}[1]{$^{#1}$}

\end{document}